\documentclass[a4paper, 12pt]{amsart} 

\usepackage{amsmath} 
\usepackage{verbatim} 
\usepackage{amsthm} 
\usepackage{amsfonts} 
\usepackage{amssymb} 
\usepackage{mathrsfs} 
\usepackage{graphicx} 
\usepackage{multicol}
\usepackage{color}
\usepackage[margin=1.35in]{geometry}
\usepackage[all]{xy}
\usepackage{pdfsync} 

\numberwithin{equation}{section} 
\theoremstyle{plain}
\newtheorem{theorem}{Theorem}[section]
\newtheorem{proposition}[theorem]{Proposition}
\newtheorem{lemma}[theorem]{Lemma}
\newtheorem{corollary}[theorem]{Corollary}

\theoremstyle{definition}
\newtheorem{definition}[theorem]{Definition}

\theoremstyle{remark}
\newtheorem{remark}[theorem]{Remark}

\newtheorem{question}[theorem]{Question}
\newtheorem{example}[theorem]{Example}

\newcommand{\Cs}{{$C^*$-algebra}}
\newcommand{\N}{\mathbb N}
\newcommand{\Z}{\mathbb Z}

\newcommand{\C}{\mathbb C}

\newcommand{\cA}{\mathcal A}

\newcommand{\cM}{\mathcal M}

\newcommand{\cK}{\mathcal K}

\newcommand{\supp}{\operatorname{supp}}
\newcommand{\Supp}{\operatorname{Supp}}
\newcommand{\red}{{\mathrm{red}}}

\newcommand{\ep}{{\varepsilon}}

\newcommand{\Type}{{\mathrm{T}}}
\newcommand{\Typ}{{\mathscr{T}}}

\newcommand{\npropto}{{\operatorname{\hskip3pt \propto \hskip-10pt /\hskip5pt}}}
\newcommand{\Can}{{\bf{K}}}
\author{Hiroki Matui}
\address{Graduate School of Science,
Chiba University, Inage-ku, Chiba 263-8522, Japan} 
\email{matui@math.s.chiba-u.ac.jp}

\author{Mikael R\o rdam}
\address{Department of Mathematics, University of Copenhagen, Univer\-si\-tets\-parken~5, DK-2100, Copenhagen \O, Denmark} 
\email{rordam@math.ku.dk}

\title[Group actions on locally compact spaces]{Universal properties of group actions on locally compact spaces}

\begin{document}

\maketitle

\begin{abstract} We study universal properties of locally compact $G$-spaces for countable infinite groups $G$. In particular we consider open invariant subsets of the $G$-space $\beta G$, and their minimal closed invariant subspaces. These are locally compact free $G$-spaces, and the latter are also minimal. We examine the properies of these $G$-spaces with emphasis on their universal properties.

As an example of our resuts, we use combinatorial methods to show that each countable infinite group admits a free minimal action on the locally compact non-compact Cantor set. 
\end{abstract}

\section{Introduction}

\noindent Ellis proved in \cite{Ellis:minimal} that every  group $G$ admits a free minimal action on a compact Hausdorff space, and he proved that universal minimal $G$-spaces exist and are unique.  A (minimal) compact $G$-space is universal if any other (minimal) compact $G$-space is the image of the universal space by a continuous $G$-map. Ellis proved, more specifically, that each minimal closed invariant subset of the $G$-space $\beta G$ is a free minimal $G$-space which is universal. The number of minimal closed invariant subsets of $\beta G$ is very large, see \cite{HLS:max-ideals}, but they are all isomorphic by Ellis' uniqueness theorem. 

Hjorth and Molberg, \cite{HjoMol:actions}, established the existence of a free minimal action of any countable infinite group on the Cantor set  (this is obtained from Ellis' results through a standard reduction argument).  They also obtained a free action of an arbitrary countable group on the Cantor set admitting an invariant Borel probability measure.  

The goal of this paper is to extend Ellis' results to the locally compact, non-compact setting, and in particular to study properties of the locally compact $G$-spaces that arise as \emph{open} invariant subsets of $\beta G$, and their minimal \emph{closed} invariant subsets (when they exist). We are particularly interested in those open invariant subsets of $\beta G$ that give rise to co-compact $G$-spaces (they always contain minimal closed invariant subsets).  These turn out to be of the form $X_A = \bigcup_{g \in G} K_{gA}$ for some subset $A$ of $G$, where $K_{gA}$ denotes the closure of the set $gA$ in $\beta G$, and they are thus "indexed" by the set $A$. If $A$ and $B$ are subsets of $G$, then $X_A = X_B$ if and only if $A$ is "$B$-bounded" and $B$ is "$A$-bounded", or, equivalently, if the Hausdorff distance between $A$ and $B$ with respect to (any) proper right-invariant metric on $G$ is finite. 

The minimal closed invariant subspaces of a co-compact open invariant subspace of $\beta G$ provide examples of locally compact free minimal $G$-spaces. We prove universality and uniqueness results for these spaces (explained in more detail below).

Kellerhals, Monod and the second named author studied  actions of \emph{supra\-me\-nable} groups on locally compact spaces in \cite{KelMonRor:supramenable}. By definition, a group is supramenable if it contains no (non-empty) paradoxical subset.  It was shown in \cite{KelMonRor:supramenable} that a group is supramenable if and only if whenever it acts 
\emph{co-compactly} on a locally compact Hausdorff space, then there is a non-zero invariant Radon measure. As a step towards proving this result it was shown that there is a non-zero invariant Radon measure on the $G$-space $X_A$ if and only if $A$ is non-paradoxical. This is an example where a property of the $G$-space $X_A$ is reflected in a property of the set $A$. We shall exploit such connections further in this paper. The condition that the action be co-compact in the characterization of supramenable groups from \cite{KelMonRor:supramenable} cannot be removed as shown in Section~\ref{sec:cocompact}.

It was further shown in \cite{KelMonRor:supramenable} that if $A$ is paradoxical, then any minimal closed invariant subset of $X_A$ is a  \emph{purely infinite}\footnote{An action of a group on a totally disconnected space is purely infinite if all compact-open subsets are paradoxical relatively to the compact-open subsets of the space.}  free minimal $G$-space. This was used to prove that any countable non-supramenable group admits a free minimal purely infinite action on  the locally compact non-compact Cantor set. 
It was left open in \cite{KelMonRor:supramenable} if \emph{all} countable infinite groups admit a free minimal action on the locally compact non-compact Cantor set. In 
Section~\ref{sec:non-discrete-compact} we answer this question affirmatively for all countable infinite groups. 

We show that the $G$-spaces $X_A$  are universal with respect to the class of locally compact $G$-spaces that have the same \emph{type} as $X_A$. The (base point dependent) type is defined for each pair $(X,x_0)$, where $X$ is a locally compact (co-compact) $G$-space such that $G.x_0$ is dense in $X$, and it  is defined to be the collection of sets $B \subseteq G$ for which $B.x_0$ is  relatively compact in $X$. In the co-compact case this information can be compressed into the equivalence class, $[A]$, of a single set $A \subseteq G$, and we say that the type of $(X,x_0)$ is $[A]$ in this case. The type of the $G$-space $(X_A,e)$ is  $[A]$, which in particular shows that all (equivalence classes of) subsets of $G$ are realized as a type. We show that there is a (necessarily unique and surjective) proper continuous $G$-map $\varphi \colon X_A \to X$ with $\varphi(e) = x_0$ if and only if the type of $X$ is $[A]$, thus providing a universal property of the space $X_A$, see Example~\ref{ex:four}, Case III.

If $A \subseteq G$ is the type of a minimal $G$-space (with respect to some base point $x_0$), then all minimal closed invariant subsets of $X_A$ are also of type $[A]$, and we show that such minimal $G$-spaces are universal among all minimal locally compact $G$-spaces $(X,x_0)$ of type $[A]$. Moreover, relying heavily on the ideas from a new proof of Ellis' uniqueness theorem by Gutman and Li, \cite{GutLi:universal}, we prove that all minimal closed invariant subspaces of $X_A$ are pairwise isomorphic as $G$-spaces whenever $A$ is of "minimal type".  If $A$ is not of minimal type, then $X_A$ may have non-isomorphic, even non-homeomorphic, closed invariant subsets.  

The minimal closed invariant subspaces of the co-compact $G$-spaces $X_A$ is a source of examples of free minimal locally compact $G$-spaces. These spaces are always totally disconnected. One can obtain a free minimal action of  the given (countable infinite) group on the locally compact non-compact Cantor set from any of these minimal closed invariant subspaces of $X_A$ using a standard reduction, provided that the minimal $G$-space is  neither compact nor discrete. 

It was shown in \cite{KelMonRor:supramenable} that a minimal $G$-space of $X_A$ is never compact if $A$ is not equivalent to an absorbing subset of $G$ (see \cite[Definition 3.5]{KelMonRor:supramenable} or Definition~\ref{def:absorbing}). It was observed in \cite{KelMonRor:supramenable} that a minimal $G$-space of $X_A$ can be discrete, even when $A$ is infinite (in which case $X_A$ itself is non-discrete), but it was left open precisely for which subsets $A$ of $G$ this can happen. 
We show here that $X_A$ contains no discrete minimal closed invariant subspace if and only if $A$ is \emph{infinitely divisible} (see Definition~\ref{def:divisible}). Accordingly, if $A$ is infinitely disivible and not equivalent to an absorbing set, then all minimal closed invariant subsets of $X_A$ are non-compact and non-discrete. We show that each countable infinite group $G$ contains an infinitely divisible subset $A$ which is not equivalent to an absorbing set, and we conclude that each countable infinite group admits a free minimal action on a non-compact non-discrete locally compact Hausdorff space, which, moreover, can be taken to be the locally compact non-compact Cantor set. 
 
As an illustration of the problems we are attempting to address, consider the following concrete question: Let $X_1$ and $X_2$ be minimal locally compact $G$-spaces. When does there exist a minimal locally compact $G$-space $Z$ and proper continuous (necessarily surjective) $G$-maps $Z \to X_j$, $j=1,2$?  If $Z$ exists and if one of $X_1$ or $X_2$ is compact, respectively, discrete, then the other must also be compact, respectively, discrete (and, conversely, in these cases $Z$ clearly does exist). We shall give a sufficient and necessary condition for the existence of the $G$-space $Z$ in the general case, where $X_1$ and $X_2$ are locally compact $G$-spaces in terms of the (base point free) type of $X_1$ and $X_2$ (Corollary~\ref{cor:char-same-type}). The base point free type is developed in Section~\ref{sec:basepoint}, where it is also used to give an alternative characterization of subsets $A$ of $G$ of minimal type, as well as to make more precise the meaning of universal minimal locally compact $G$-spaces.

\subsection*{Acknowledgements} This research was done in 2013
while the first named author visited
Department of Mathematical Sciences, University of Copenhagen.
He gratefully appreciates the warm hospitality of the Department.
He was supported in part
by the Grant-in-Aid for Young Scientists (B)
of the Japan Society for the Promotion of Science.

The second named author thanks David Kerr, Nicolas Monod and Zhuang Niu for inspiring conversations relating to the topics of this paper. He was supported  by the Danish National Research Foundation (DNRF) through the Centre for Symmetry and Deformation at University of Copenhagen, and The Danish Council for Independent Research, Natural Sciences. Both authors thank the anonymous referee for several useful suggestions and for supplying a proof of Lemma~\ref{lm:I_n} in the present general case.


\section{Preliminaries} \label{sec:prelim}

\noindent This section contains some background material, primarily from \cite{KelMonRor:supramenable}. 

Throughout this paper $G$ will denote a (discrete) group with neutral element $e \in G$. Most of the time $G$ will be assumed to be countable and infinite. 
We shall be studying locally compact (typically non-compact) $G$-spaces, and we shall primarily be interested in $G$-spaces with the following property:

\begin{definition}[Co-compact actions] \label{def:co-cpt} Let $G$ be a  group acting on a locally compact Hausdorff space $X$. The action of $G$ on $X$ (or the $G$-space $X$) is said to be \emph{co-compact} if there is a compact subset $K$ of $X$ such that
$$\bigcup_{g \in G} g.K = X.$$
\end{definition}

\noindent Every \emph{minimal} locally compact $G$-space is co-compact (take $K$ to be any compact set with non-empty interior). 

In each  locally compact co-compact Hausdorff $G$-space $X$ there is a compact subset $K$ of $X$ such that $X = \bigcup_{g \in G} g.K^{\mathrm{o}}$. A compact subset $K$ with these properties will be called \emph{$G$-regular}. 

Here is a useful property of co-compact $G$-spaces:

\begin{proposition} \label{co-cpt-minimal}
Each locally compact co-compact $G$-space $X$ contains a minimal closed $G$-invariant subset. Moreover, one can find such a minimal closed $G$-invariant subset inside every non-empty closed $G$-invariant subset of $X$. 
\end{proposition}

\begin{proof}  We pass to complements and show that each $G$-invariant open proper subset $U$ of $X$ is contained in a maximal $G$-invariant open proper subset of $X$. Use Zorn's lemma to find a maximal linearly ordered collection $\{U_\alpha\}_{\alpha \in I}$ of $G$-invariant open proper subset of $X$ each containing $U$, and set $V = \bigcup_{\alpha \in I} U_\alpha$. Then $V$ is an open $G$-invariant subset of $X$ which contains $U$ and is not properly contained in any $G$-invariant open proper subset of $X$. We must show that $V \ne X$. 

Let $K$ be a compact subset of $X$ that witnesses the co-compactness. If $V = X$, then $K \subseteq U_\alpha$ for some $\alpha \in I$. But this would imply that $X = \bigcup_{g \in G} g.K \subseteq U_\alpha$, contradicting that $U_\alpha \ne X$ for all $\alpha \in I$. 
\end{proof}

\noindent
The following definition (also considered in \cite{KelMonRor:supramenable}) plays a central role in this paper:

\begin{definition} \label{def:G-equiv}
Let $G$ be a group. Let $A$ and $B$ be non-empty subsets of $G$. Write $A \propto B$ if $A$ is \emph{$B$-bounded}, i.e., if $A \subseteq FB = \bigcup_{g \in F} gB$ for some finite subset $F$ of $G$. Write $A \approx B$ if $A \propto B$ and $B \propto A$. 

Denote by $P_\approx(G)$ the set of equivalence classes $P(G)/\!\! \approx$. 
\end{definition} 

\noindent It is easy to see that  $\propto$ is a pre-order relation on the power set $P(G)$ of $G$, and $\approx$ is an equivalence relation on $P(G)$.  The relation $\propto$ defines a partial order relation on $P_\approx(G)$, which again is denoted by $\propto$. When, in the sequel, we call two subsets of a group \emph{equivalent}, we shall have the equivalence relation $\approx$ defined above in mind.

\begin{example} \label{ex:subgroup1} (i). Let $G$ be a group and let $H$ be a subgroup of $G$. Then $G \approx H$ if and only if $|G:H| < \infty$. 

(ii). More generally, if $H_1 \subseteq H_2 \subseteq G$ are subgroups of a group $G$, then $H_1 \approx H_2$ if and only if $|H_2:H_1| < \infty$.

(iii). Let $F \subseteq G$ be non-empty. Then $F \approx \{e\}$ if and only if $F$ is finite. 
\end{example}

\noindent We have the following geometric interpretation of the relations defined above. 

\begin{lemma} \label{lm:approx-geom}
Let $G$ be a  group equipped with a proper right-invariant metric\footnote{A metric $d$ on $G$ is said to be \emph{right-invariant} if $d(hg,h'g) = d(h,h')$ for all $g,h,h' \in G$. It is \emph{proper} if the set $\{g \in G : d(g,h) \le R\}$ is finite for all $h \in G$ and all $R < \infty$. Every countable group admits a proper right-invariant metric.}  $d$, and let $A$ and $B$ be non-empty subsets of $G$. Then:
\begin{enumerate}
\item $A \propto B$ if and only if there exists $R < \infty$ such that $d(g,B) \le R$ for all $g \in A$. \vspace{.1cm}
\item $A \approx B$ if and only if there exists $R < \infty$ such that $d(g,B)\le R$ and $d(h,A) \le R$ for all $g \in A$ and $h \in B$, i.e., if the Hausdorff distance between $A$ and $B$ with respect to $d$ is finite. 
\end{enumerate}
\end{lemma}

\noindent Note that the conclusion of the lemma does not depend on the choice of proper right-invariant metric. 

\begin{proof} (i). Suppose that $A \subseteq FB$ for some finite subset $F$ of $G$. Set
$$R =  \max\{d(g,e) \mid g \in F\} < \infty.$$
Let $g \in A$ be given. Then $g=th$ for some $t \in F$ and $h \in B$, so $d(g,B) \le d(g,h) = d(t,e) \le R$. 

Assume conversely that $d(g,B) \le R$ for all $g \in A$.  Let $F$ be the set of elements $t \in G$ such that $d(t,e) \le R$. Let $g \in A$. Then $d(g,h) \le R$ for some $h \in B$, whence $gh^{-1} \in F$, so $g \in FB$. This proves that $A \subseteq FB$. 

(ii) follows from (i).
\end{proof}

\noindent It is perhaps of interest to note that the relations from Definition \ref{def:G-equiv} also can be interpreted at the level of $C^*$-algebras, more precisely, in terms of the Roe algebra $\ell^\infty(G) \rtimes_\red G$. The Roe algebra  is the sub-\Cs{} of $B(\ell^2(G))$ generated by the natural copy of $\ell^\infty(G)$ inside $B(\ell^2(G))$ and the image, $\lambda(G)$, of the left-regular representation, $\lambda$, of $G$ on $\ell^2(G)$. Denote the unitary operator $\lambda(g)$ by $u_g$. Then 
$$(u_g  f  u_g^*)(h) = f(g^{-1}h), \qquad u_g \, 1_A \, u_g^* = 1_{gA},$$
for all $f \in \ell^\infty(G)$, all $g,h \in G$, and all subsets $A$ of $G$. 

Let $E \colon \ell^\infty(G) \rtimes_\red G \to \ell^\infty(G)$ be the canonical conditional expectation. 

\begin{lemma} \label{lm:1_A}
Let $G$ be a discrete group and let $A$ and $B$ be non-empty subsets of $G$. Then:
\begin{enumerate}
\item $A \propto B$ if and only if $1_A$ is in the closed two-sided ideal in $\ell^\infty(G) \rtimes_\red G$ generated by $1_B$. \vspace{.1cm}
\item $A \approx B$ if and only if $1_A$ and $1_B$ generate the same closed two-sided ideal in $\ell^\infty(G) \rtimes_\red G$. 
\end{enumerate}
\end{lemma}

\begin{proof} (i). Each element $y \in \ell^\infty(G) \rtimes_\red G$ can be written as a (formal) sum $y = \sum_{g \in G} f_g \, u_g$; and $E(y) = f_e$. Let $\Supp(y)$ be the set of those $g \in G$ such that $f_g \ne 0$. The set of elements $y$ for which $\Supp(y)$ is finite is a dense $^*$-subalgebra of $\ell^\infty(G) \rtimes_\red G$. 

If $f \in \ell^\infty(G)$, then let $\supp(f)$ be the set of those $g\in G$ such that $f(g) \ne 0$. An easy calculation shows that for $y \in \ell^\infty(G) \rtimes_\red G$, with $\Supp(y) \, (= F)$ finite and $A \subseteq G$, one has
\begin{equation} \label{eq:supp-1}
\supp E(y \, 1_A \, y^*) \subseteq FA.
\end{equation}

Suppose that $1_A$ belongs to the closed two-sided ideal in $\ell^\infty(G) \rtimes_\red G$ generated by $1_B$. Then there exist $x_1, \dots, x_n \in \ell^\infty(G) \rtimes_\red G$ such that $1_A = \sum_{j=1}^n x_j \, 1_B \, x_j^*$. We can approximate each $x_j$ with an element $y_j$ in $\ell^\infty(G) \rtimes_\red G$, with $\Supp(y_j) \, (= F_j)$ finite, and such that 
$$\|1_A - \sum_{j=1}^n y_j \, 1_B \, y_j ^*\| < 1.$$
Then $\|1_A - \sum_{j=1}^n E(y_j \, 1_B \, y_j^*) \| < 1$, so 
$$A \; \subseteq \; \bigcup_{j=1}^n \supp(E(y_j\, 1_B \, y_j^*) ) \; \subseteq \; \bigcup_{j=1}^n F_jB \; = \; FB,$$
by \eqref{eq:supp-1}, when $F = \bigcup_{j=1}^n F_j$. This shows that $A \propto B$.

Suppose now that $A \propto B$, and let $F \subseteq G$ be a finite set such that $A \subseteq FB$. Then
$$1_A \le 1_{FB} \le \sum_{g \in F} 1_{gB} = \sum_{g \in F} u_g \, 1_B \, u_g^*.$$
The element on the right-hand side belongs to the closed two-sided ideal in the Roe algebra generated by $1_B$, and hence so does $1_A$.

(ii) clearly follows from (i).
\end{proof}

\vspace{.3cm} \noindent {\bf{The $G$-space $\beta G$ and its open invariant subsets.}}

\vspace{.2cm} \noindent
Any (discrete) group $G$ acts on itself by left multiplication. By the universal property of the $\beta$-compactification this action extends to a continuous action of $G$ on $\beta G$. The action of $G$ on $\beta G$ is free; and it is amenable if and only if $G$ is exact (see \cite[Theorem 5.1.6]{BroOza:book}). The $G$-space $\beta G$ is never minimal (unless $G$ is finite). In fact, $\beta G$ has an abundance of open invariant subsets whenever $G$ is infinite (see \cite{HLS:max-ideals}). 

Note that $\beta G$ has a dense orbit (for example $G = G.e \subseteq \beta G$, where $e \in G$ is the neutral element). Each non-empty open invariant subset of $\beta G$ contains $G$ as an open and dense subset, and is thus a locally compact $G$-space with a dense orbit. In Proposition \ref{prop:dense-orbit} below we shall say more about $G$-spaces with a dense orbit. 

We proceed to describe the open invariant subsets of $\beta G$. Let $A$ be a non-empty subset of $G$, and, following the notation of \cite{KelMonRor:supramenable}, let $K_A$ denote the closure of $A$ in $\beta G$. Then $K_A$ is a compact-open subset of $\beta G$ (since $\beta G$ is a Stonean space). Put
\begin{equation} \label{eq:X_A}
X_A = \bigcup_{g \in G} g.K_A = \bigcup_{g \in G} K_{gA} \subseteq \beta G.
\end{equation}
Then $X_A$ is an open and invariant subset of $\beta G$, and is hence a locally compact  $G$-space, which is also co-compact and has a dense orbit $G = G.e$. 

It is observed in \cite[Lemma 2.5]{KelMonRor:supramenable} that $X_A = X_B$ if and only if $A \approx B$, and $X_A \subseteq X_B$ if and only if $A \propto B$. In particular, $X_A = G$ if and only if $A$ is finite and non-empty, and $X_A = \beta G\; (= X_G)$ if and only if $A \approx G$.

\begin{proposition}  \label{prop:X_A} Each $G$-invariant open subset of $\beta G$ on which $G$ acts co-compactly is equal to $X_A$ for some non-empty $A \subseteq G$.
\end{proposition}

\begin{proof} As remarked below Definition \ref{def:co-cpt} there is a $G$-regular compact subset $K$ of $X$. Let $L$ be the closure of $K^\mathrm{o}$. Then $L \subseteq K \subseteq X$, and $L$ is compact-open (being equal to the closure in $\beta G$ of the open set $K^\mathrm{o}$). Hence $L = K_A$, when $A = L \cap G$, cf.\ \cite[Lemma 2.4]{KelMonRor:supramenable}. Since $K_A \subseteq X$ and $X$ is $G$-invariant, it follows that $X_A \subseteq X$. Conversely,
$$X = \bigcup_{g \in G} g.K^{\mathrm{o}} \subseteq \bigcup_{g \in G} g.K_A = X_A.$$
\end{proof}

\noindent We conclude from Proposition \ref{prop:X_A} and the previous  remarks  that the map $A \mapsto X_A$ induces an order isomorphism from the partially ordered set $(P_\approx(G), \propto)$ onto the set of $G$-invariant open co-compact subsets of $\beta G$. 

Not all open invariant subsets of $\beta G$ are co-compact as $G$-spaces as will be shown in Proposition \ref{prop:non-co-cpt}. However, they can still be classified in terms of \emph{left $G$-ideals}, see Proposition~\ref{prop:X_M}.

We end this section with a description of $G$-spaces with a dense orbit. The proposition below is probably well-known to experts. The second named author thanks  Zhuang Niu for pointing out that the implication (iii) $\Rightarrow$ (vi) holds (in the second countable case).

A $G$-space $X$ is said to be \emph{topologically transitive} if for every pair of non-empty open sets $U$ and $V$ there is $g \in G$ such that $g.U \cap V \ne \emptyset$. An action of $G$ on $X$ is said to have the \emph{intersection property} if each non-zero ideal in $C_0(X) \rtimes_\red G$ has non-zero intersection with $C_0(X)$. It is well-known, and follows for example from \cite[Lemma 7.1]{OlePed:C*-dynamicIII} (as well as from the work of Elliott and Kishimoto), that an action of $G$ on $X$ has the intersection property if  it is topologically free. (If the action of $G$ on $X$ is topologically free, then the associated action, $g \mapsto \alpha_g$, of $G$ on $C_0(X)$ consists of properly outer automorphisms, $\alpha_g$, for $g \ne e$, and then we can apply \cite[Lemma 7.1]{OlePed:C*-dynamicIII} to conclude that the action has the intersection property.)

\begin{proposition} \label{prop:dense-orbit}
Consider the following conditions on a countable group $G$ acting on a locally compact Hausdorff space $X$:
\begin{enumerate}
\item $C_0(X) \rtimes_{\mathrm{red}} G$ is prime. \vspace{.1cm}
\item The action of $G$ on $X$ is topologically transitive. \vspace{.1cm}
\item Each non-empty open $G$-invariant subset of $X$ is dense in $X$. \vspace{.1cm}
\item Each proper closed $G$-invariant subset of $X$ has empty interior.  \vspace{.1cm}
\item There exists $x \in X$ such that $G.x$ is dense in $X$. \vspace{.1cm}
\item The set of points $x \in X$ for which $G.x$ is dense in $X$ is a dense $G_\delta$-set.
\end{enumerate}
Then
$$ \mathrm{(i)} \Rightarrow \mathrm{(ii)} \Leftrightarrow \mathrm{(iii)} \Leftrightarrow \mathrm{(iv)}   \Leftarrow \mathrm{(v)} \Leftarrow \mathrm{(vi)};$$
and  {\rm{(iv)}} $\Rightarrow$ {\rm{(vi)}} if $X$ is second countable.
Moreover, {\rm{(ii)}} $\Rightarrow$ {\rm{(i)}} if the action of $G$ on $X$ has the intersection property.

In particular, all six conditions are equivalent if $X$ is second countable and the action of $G$ on $X$ is topologically free.
\end{proposition}

\begin{proof} The implications (iii) $\Leftrightarrow$ (iv) and (vi) $\Rightarrow$ (v) are trivial.

(i) $\Rightarrow$ (ii). Let $U$ and $V$ be open non-empty subsets of $X$,  and choose non-zero positive elements $a \in C_0(U) \subseteq C_0(X)$ and $b \in C_0(V) \subseteq C_0(X)$. Since  $C_0(X) \rtimes_{\mathrm{red}} G$ is prime it follows that 
$$a \Big(C_0(X) \rtimes_{\mathrm{red}} G\Big) b \ne 0.$$
This implies that $au_gb \ne 0$ for some $g \in G$, where $g \mapsto u_g$ is the unitary representation of $G$ in (the multiplier algebra of) $C_0(X) \rtimes_{\mathrm{red}} G$. 

The support of  the non-zero element $au_gbu_g^* \in C_0(X)$ is contained in $U \cap g.V$, so $U \cap g.V \ne \emptyset$.

(ii) $\Rightarrow$ (iii). Let $U$ be a non-empty open invariant subset of $X$. If $U$ is not dense, then there is an open (not necessarily invariant) subset $V$ of $X \setminus U$. As $g.U = U$ for all $g \in G$ there is no $g \in G$ such that $g.U \cap V \ne \emptyset$. 

(iii) $\Rightarrow$ (ii). Let $U$ and $V$ be non-empty open subsets of $X$. Then $\widetilde{U} = \bigcup_{g \in G} g.U$ is a (non-empty) invariant open subset of $X$. If (iii) holds, then $\widetilde{U} \cap V \ne \emptyset$. This implies that $g.U \cap V \ne \emptyset$ for some $g\in G$. Hence (ii) holds.

(v) $\Rightarrow$ (iii). Let $x \in X$ be such that $G.x$ is dense in $X$, and let $U$ be a non-empty invariant open subset of $X$. Then $g.x \in U$ for some $g \in G$. As $U$ is invariant it follows that $G.x \subseteq U$. Hence $U$ must be dense. 

We prove  (iii) $\Rightarrow$ (vi) assuming that $X$ is second countable. Let $\{U_n\}_{n = 1}^\infty$ be a basis for the topology on $X$ consisting of (non-empty) open sets, and put $\widetilde{U}_n = \bigcup_{g \in G} g.U_n$. Then each $\widetilde{U}_n$ is non-empty, $G$-invariant and open, and therefore dense in $X$. It follows that 
$$X_0 := \bigcap_{n \in \N} \widetilde{U}_n$$
is a dense $G_\delta$ set in $X$. If $x \in X_0$, then $G.x \cap U_n \ne \emptyset$ for all $n \in \N$. This shows that $G.x$ is dense in $X$ for all $x \in X_0$. 

Finally, we prove  (iii) $\Rightarrow$ (i) assuming that the action of $G$ on $C_0(X)$ has the intersection property. Let $I$ and $J$ be closed two-sided non-zero ideals in $C_0(X) \rtimes_\red G$. Then $I_0 = I \cap C_0(X)$ and $J_0 = J \cap C_0(X)$ are non-zero. Hence $I_0 = C_0(U)$ and $J_0 = C_0(V)$ for some non-empty open invariant subsets $U$ and $V$ of $X$. It follows that $U \cap V \ne \emptyset$. Hence $I_0 J_0 \ne 0$, so also $IJ \ne 0$.
\end{proof}


\section{The type of a group action on a locally compact space}  \label{sec:type}

\noindent In Section~\ref{sec:X_A} we will show that the locally compact $G$-space $X_A$ associated to a subset $A$ of $G$, considered in the previous section, has a universal property  in a similar  way as the space $\beta G$ itself is universal among all compact $G$-spaces possesing a dense orbit. In Section~\ref{sec:min-universal} we will determine the universal properties of the minimal closed $G$-invariant subspaces of $X_A$ among minimal locally compact $G$-spaces. It turns out that $X_A$ is universal relatively to a sub-class of the locally compact $G$-spaces (with a dense orbit), namely those that have the same \emph{type} as $X_A$ itself. 

In this section we shall define and prove basic properties of the type of an action of a group on a locally compact Hausdorff space as alluded to above. The type keeps track of which subsets of the group gives rise to relatively compact subsets of the space. To make sense of this we must specify a base point of the space in a dense orbit, and we talk about a \emph{pointed} locally compact $G$-space.  In Section \ref{sec:basepoint} we shall define the type of a locally compact $G$-space without reference to a base point. 

The type will be defined in terms of a left $G$-ideal in the power set of the group. 

\begin{definition} \label{def:leftideal}
Let $G$ be a group. A collection $\cM$ of subsets of $G$ is called a \emph{left $G$-ideal} if it contains at least one non-empty set and
\begin{enumerate}
\item $A \in P(G)$, $B \in \cM$, and $A \subseteq B$ implies $A \in \cM$, \vspace{.1cm}
\item $A, B \in \cM$ implies $A \cup B \in \cM$, \vspace{.1cm}
\item $A \in \cM$ and $g \in G$ implies $gA \in \cM$.
\end{enumerate}

A left $G$-ideal $\cM$ is said to be \emph{compact} if it is compact in the standard hull-kernel topology on the set of all left $G$-ideals. In other words, $\cM$ is compact if and only if whenever $\{\cM_\alpha\}_{\alpha \in I}$ is an upwards directed net of left $G$-ideals such that $\cM \subseteq \bigcup_{\alpha \in I} \cM_\alpha$, then $\cM \subseteq \cM_\alpha$ for some $\alpha \in I$. 
\end{definition}

\noindent Recall the definition of the order relation "$\propto$" from Definition~\ref{def:G-equiv}. For emphasis we mention the following (trivial) fact about this order relation and left $G$-ideals:

\begin{lemma} \label{lm:A->B} Let $G$ be a group, let $A,B$ be non-empty subsets of $G$, and let $\cM$ be a left $G$-ideal. If $A \propto B$ and $B \in \cM$, then $A \in \cM$.
\end{lemma}

\begin{example} 
(a). For each fixed non-empty subset $A$ of $G$ set $$\cM_A = \{B \in P(G) \mid B \propto A\}.$$ 
Then $\cM_A$ is a left $G$-ideal. By transitivity of the relation "$\propto$", we see that if $A$ and $B$ are subsets of $G$, then $\cM_A = \cM_B$ if and only if $A \approx B$. 

(b). The collection $\cM = P(G)$ of all subsets of $G$ is a left $G$-ideal. It is equal to $\cM_G$, and is the largest left $G$-ideal. 

(c). The collection $\cM_{\mathrm{fin}}$ of all finite subsets of $G$ is a left $G$-ideal. It is equal to $\cM_{\{e\}}$, cf.\ Example \ref{ex:subgroup1} (iii), and it is the smallest left $G$-ideal.
\end{example}

\begin{proposition} \label{prop:M=M_A}
Let $\cM$ be a left $G$-ideal. Then $\cM$ is compact if and only if $\cM= \cM_A$ for some $A \subseteq G$.
\end{proposition}

\begin{proof} Let $\emptyset \ne A \subseteq G$. If $\cM_A \subseteq \bigcup_{\alpha \in I} \cM_\alpha$ for some increasing net of left $G$-ideals, then $A \in \cM_\alpha$ for some $\alpha \in I$, whence $\cM_A \subseteq \cM_\alpha$ by Lemma~\ref{lm:A->B}. This shows that $\cM_A$ is compact. 

Suppose that $\cM$ is compact. The family $\{\cM_A\}_{A \in \cM}$ is an upward directed net of left $G$-ideals which satisfies $\cM = \bigcup_{A \in \cM} \cM_A$. It follows by compactness that $\cM = \cM_A$ for some $A \in \cM$. 
\end{proof}

\begin{definition} Let $G$ be a group. By a \emph{pointed locally compact $G$-space} we shall mean a pair $(X,x_0)$ consisting of a locally compact Hausdorff space $X$ on which the group $G$ acts, and a point $x_0 \in X$ such that $G.x_0$ is dense in $X$. 

To each pointed locally compact $G$-space $(X,x_0)$ associate the set
$$\cM(G,X,x_0) = \{ A \in P(G) \mid \overline{A.x_0} \: \,  \text{is compact}\}.$$
\end{definition}

\vspace{.2cm}
\noindent It is easy to verify that the invariant $\cM(G,X,x_0)$  is a left $G$-ideal.  

If $X$ is a $G$-space, then associate to each $x \in X$ and to each subset $V \subseteq X$ the following subset of $G$:
\begin{equation} \label{eq:O(V,x)}
O_X(V,x) = \{g \in G \mid g.x \in V\}.
\end{equation}

The results of the following two lemmas will be used very often. The easy proof of the first lemma is omitted.

\begin{lemma} \label{lm:O(V,x)-1}
Let $X$ be a $G$-space, let $V \subseteq X$, let $x \in X$, and set $A = O_X(V,x)$. 
\begin{enumerate}
\item If $F$ is a subset of $G$, then $O_X(F.V,x) = FA$. \vspace{.1cm}
\item $G.x \cap V = A.x$. In particular, $A.x \subseteq V$.
\end{enumerate}
\end{lemma}

\noindent
Recall the definition of a $G$-regular compact set from below Definition~\ref{def:co-cpt}.

\begin{lemma} \label{lm:G-regular}
Suppose that $(X,x_0)$ is a pointed locally compact $G$-space. Let $V \subseteq X$ and set $A = O_X(V,x_0)$. Then:
\begin{enumerate}
\item $V^{\mathrm{o}} \subseteq \overline{A.x_0}$. 
 \vspace{.1cm}
\item $A \in \cM(G,X,x_0)$ if $V$ is relatively compact. 
 \vspace{.1cm}
\item If $X$ is a co-compact $G$-space and if $K$ is a $G$-regular compact subset of $X$, then for any compact subset $L$ of $X$ and for any $x \in X$, we have
$$O_X(L,x) \; \propto \; O_X(K^\mathrm{o},x) \; \subseteq \; O_X(K,x).$$
\end{enumerate}
\end{lemma}

\begin{proof} (i) follows from Lemma \ref{lm:O(V,x)-1} (ii) and the fact that $G.x_0 \cap V^{\mathrm{o}}$ is dense in $V^{\mathrm{o}}$ when $G.x_0$ is dense in $X$. 
(ii) follows Lemma \ref{lm:O(V,x)-1} (ii) and from the definition of $\cM(G,X,x_0)$.

(iii). By $G$-regularity of $K$ and compactness of $L$ we find that $L \subseteq \bigcup_{g \in F} g.K^{\mathrm{o}} = F.K^{\mathrm{o}}$ for some finite subset $F$ of $G$. Hence (iii) holds by Lemma \ref{lm:O(V,x)-1} (i). 
\end{proof}

\noindent The invariant $\cM(G,X,x_0)$ depends on $x_0$ in a subtle way (see more about this in Section \ref{sec:basepoint}). Some properties of this invariant, however, are independent of the choice of base point, such as when the left $G$-ideal  $\cM(G,X,x_0)$ is compact. 

\begin{proposition} \label{prop:co-cpt}
Let $G$ be a countable group and let $(X,x_0)$ be a pointed locally compact $G$-space. It follows that  $G$ acts co-compactly on $X$ if and only if $\cM(G,X,x_0)$ is compact and $X$ is $\sigma$-compact.
\end{proposition}

\begin{proof} Suppose first that $G$ acts co-compactly on $X$. Then $X$ is $\sigma$-compact (because $G$ is assumed to be countable). Choose a $G$-regular compact subset $K$ of $X$, and set $A = O_X(K,x_0)$. We show that $\cM(G,X,x_0) = \cM_A$, from which we can conclude that $\cM(G,X,x_0)$ is compact by Proposition~\ref{prop:M=M_A}. 

First, $A \in \cM(G,X,x_0)$ by Lemma \ref{lm:G-regular} (ii),  so $\cM_A \subseteq \cM(G,X,x_0)$. 
Suppose next that $B \in \cM(G,X,x_0)$, and let $K'$ be the closure of $B.x_0$.  Then $K'$ is compact, and we  deduce  from Lemma \ref{lm:O(V,x)-1} (ii) and Lemma \ref{lm:G-regular} (iii) that 
$$B \; \subseteq \; O_X(K',x_0) \; \propto \; O_X(K,x_0)=A,$$
so $B \in \cM_A$.   This shows that $\cM(G,X,x_0) \subseteq \cM_A$.

Suppose now that $\cM(G,X,x_0)$ is compact and $X$ is $\sigma$-compact. Then $\cM(G,X,x_0) = \cM_A$ for some $A \subseteq G$ by Proposition~\ref{prop:M=M_A}, and $K := \overline{A.x_0}$  is compact. We show that $K$ witnesses co-compactness of the action of $G$ on $X$.
Find a sequence $\{L_n\}_{n=1}^\infty$ of compact subsets of $X$ such that $L_n \subseteq L_{n+1}^\mathrm{o}$ for all $n$, and $X = \bigcup_{n=1}^\infty L_n$. Put $B_n = O_X(L_n,x_0)$. Then $B_n \in \cM(G,X,x_0) = \cM_A$, so $B_n \subseteq F_nA$ for some finite subset $F_n$ of $G$. 

We now have:
\begin{eqnarray*}
L_n^{\mathrm{o}} \; \subseteq \; \overline{B_n.x_0} \; \subseteq \; \overline{F_nA.x_0} \;=\; \bigcup_{g \in F_n} \overline{gA.x_0} \; = \; \bigcup_{g \in F_n} g.K \; \subseteq \; \bigcup_{g \in G} g.K
\end{eqnarray*}
for all $n \ge 1$, where the first inclusion follows from Lemma \ref{lm:G-regular} (i). This shows that  $\bigcup_{g \in G} g.K = X$.
\end{proof}

\begin{definition}[The type of a pointed locally compact $G$-space] \label{def:type}
Let $G$ be a group, let $P_\approx(G)$ be as in Definition \ref{def:G-equiv}, and for each $A \subseteq G$ let $[A] \in P_\approx(G)$ denote the equivalence class containing $A$.

To each pointed locally compact co-compact $G$-space $(X,x_0)$ set 
$$\Type(G,X,x_0) =[A] \in P_\approx(G),$$ 
whenever $\emptyset \ne A \subseteq G$ is such that $\cM(G,X,x_0) = \cM_A$. This is well-defined by  Proposition \ref{prop:co-cpt} and by the fact that $\cM_A = \cM_B$ if and only if $A \approx B$. 

We call $[A]$ the \emph{type} of $(X,x_0)$, and we say that $A$ \emph{represents the type of $(X,x_0)$}. 
\end{definition}

\noindent One can define the type of a not necessarily co-compact pointed locally compact $G$-space $(X,x_0)$ to be the left $G$-ideal $\cM(G,X,x_0)$. However, we shall not consider this general case systematically in this paper.  In the following two propositions we give conditions on a pointed locally compact co-compact $G$-space to be of a given type. The results  sharpen the statement in Proposition \ref{prop:co-cpt}.

\begin{proposition} \label{prop:regular}
Let $G$ be a countable group,  let $(X,x_0)$ be a pointed locally compact co-compact $G$-space. 
\begin{enumerate}
\item Let $A$ be a non-empty subset of $G$. Then $A$ represents the type of $(X,x_0)$ if and only if \vspace{.1cm}
\begin{enumerate}
\item $\overline{A.x_0}$ is compact,  and \vspace{.1cm}
\item $O_X(K',x_0) \propto A$ for all compact subset $K' \subseteq X$. \vspace{.1cm}
\end{enumerate}
\item For each subset $A$ of $G$ that represents the type of $(X,x_0)$, the set $K = \overline{A.x_0}$ is compact and   \vspace{.1cm}
\begin{itemize}
\item[(c)]  $\displaystyle{\bigcup_{g \in G} g.K = X}$.
\end{itemize}
\item For some subset $A$ of $G$ that represents the type of $(X,x_0)$, the set $K = \overline{A.x_0}$ is $G$-regular, i.e., $K$ is compact and  \vspace{.1cm}
\begin{itemize}
\item[(d)] $\displaystyle{\bigcup_{g \in G} g.K^\mathrm{o} = X}$. 
\end{itemize}
\end{enumerate}
\end{proposition}

\begin{proof} (i). Suppose that $\Type(G,X,x_0) = [A]$, i.e., that $\cM(G,X,x_0) = \cM_A$. As $A \in \cM_A$, we see that (a) holds. Let $K'$ be a compact subset of $X$. Then $O_X(K',x_0)$ belongs to $\cM_A$ (by Lemma \ref{lm:G-regular} (ii)), so (b) holds. 

Suppose, conversely, that (a) and (b) hold. Then $A\in\cM(G,X,x_0)$ by (a) and the definition of $\cM(G,X,x_0)$, and so $\cM_A\subseteq \cM(G,X,x_0)$. 
If $B$ is in $\cM(G,X,x_0)$, then $K'=\overline{B.x_0}$ is compact, whence
$B \subseteq O_X(K',x_0) \propto A$ by (b), so  
$B\in\cM_A$. 
We conclude that $\cM(G,X,x_0)=\cM_A$, so $\Type(G,X,x_0) = [A]$. 

(ii). Let $x \in X$ and find a relatively compact open neighbourhood $V$ of $x$.  Then $B := O_X(V,x_0)$
belongs to $\cM_A$ by Lemma \ref{lm:G-regular} (ii), so $B \propto A$. There is a finite subset $F$ of $G$ such that $B \subseteq FA$. It follows from Lemma \ref{lm:G-regular} (i) that 
$$x \; \in \; V \; \subseteq \; \overline{B.x_0} \, = \, \bigcup_{g \in F} \overline{gA.x_0} \, = \, \bigcup_{g \in F} g.K \, \subseteq \, \bigcup_{g \in G} g.K.$$
This shows that (c) holds. 

(iii). Let $B \subseteq G$ be any representative of $\Type(G,X,x_0)$ and choose a $G$-regular compact subset $L$ of $X$. Put $A_0 = O_X(L,x_0)$. Then  $A_0 \propto B$ by (b). Put $A = A_0 \cup B$ and let $K$ be the closure of $A.x_0$. Then $A \approx B$, so $\Type(G,X,x_0) = [A]$.  It follows from Lemma \ref{lm:G-regular} (i) that $L^{\mathrm{o}} \subseteq  \overline{A_0.x_0}  \subseteq K$.
This implies that $L^{\mathrm{o}} \subseteq K^{\mathrm{o}}$, so $K$ is $G$-regular. 
\end{proof}

\begin{proposition} \label{prop:type-condition}
Let $G$ be a countable group,  let $(X,x_0)$ be a pointed locally compact co-compact $G$-space and let $K$ be a $G$-regular compact subset of $X$. Then 
$$\Type(G,X,x_0) = [O_X(K^{\mathrm{o}},x_0)] = [O_X(K,x_0)] .$$
\end{proposition}

\begin{proof} Let $A\subseteq G$ be a representative of the type of $(X,x_0)$. Then
$$O_X(K^{\mathrm{o}},x_0) \subseteq O_X(K,x_0) \propto A$$
by Proposition \ref{prop:regular} (i). 
Let $L$ be the closure of $A.x_0$, which is a compact set by Proposition \ref{prop:regular} (i). Then $A  \subseteq  O_X(L,x_0) \propto O_X(K^{\mathrm{o}},x_0)$ by Lemma \ref{lm:G-regular} (iii). This completes the proof. 
\end{proof}

\noindent
Each element in $P_\approx(G)$ is the type of some, in fact a canonical, pointed locally compact co-compact $G$-space:

\begin{proposition} \label{prop:type(X_A)}
Let $A$ be a non-empty subset of a group $G$, and let $X_A \subseteq \beta G$ be the locally compact co-compact $G$-space defined in \eqref{eq:X_A}. Then $\Type(G,X_A,e) = [A]$. 
\end{proposition}

\begin{proof} We must show that $\cM(G,X_A,e) = \cM_A$. Let $B \subseteq G$. The closure of $B = B.e$ in $X_A$ is equal to $K_B \cap X_A$ (where $K_B$ is the closure of $B$ in $\beta G$). We now have
\begin{eqnarray*}
B \in \cM(G,X_A,e) & \iff & K_B \cap X_A \; \text{is compact} \\
&\overset{(1)}{\iff} & K_B \subseteq X_A \overset{(2)}{\iff} B \propto A \iff B \in \cM_A.
\end{eqnarray*}
The "$\Rightarrow$" part of (1)  follows because  $K_B$ is compact-open and $X_A$ is open and dense in $\beta G$, so $K_B \cap X_A$ is dense in $K_B$.

(2) follows from \cite[Lemma 2.5(i)]{KelMonRor:supramenable}.
\end{proof}

\noindent Suppose that $(X,x_0)$ and $(Y,y_0)$ are pointed locally compact $G$-spaces and that $\varphi \colon X \to Y$ is a continuous proper $G$-map such that $\varphi(x_0) = y_0$. Then $\varphi$ is necessarily surjective because $\varphi(G.x_0) = G.y_0$ and because any continuous proper map between locally compact Hausdorff spaces maps closed sets to closed sets. 

\begin{proposition} \label{prop:sametype}
Let $G$ be a group and let $(X,x_0)$ and $(Y,y_0)$ be pointed locally compact $G$-spaces. Suppose there exists a continuous proper $G$-map $\varphi \colon X \to Y$ such that $\varphi(x_0) = y_0$. Then $$\cM(G,X,x_0) = \cM(G,Y,y_0).$$
In particular, if one of the $G$-spaces $X$ and $Y$ is co-compact, then so is the other, in which case
 $$\Type(G,X,x_0) = \Type(G,Y,y_0).$$
\end{proposition}

\begin{proof} Let $A \subseteq G$. Observe that
$$\varphi(A.x_0) = A.y_0, \qquad A.x_0 \subseteq \varphi^{-1}(A.y_0).$$
If $A \in \cM(G,X,x_0)$, then  $A.y_0$ is contained in the compact set $\varphi(\overline{A.x_0})$, so $A \in \cM(G,Y,y_0)$. Similarly, if  $A \in \cM(G,Y,y_0)$, then $A.x_0$ is contained in the compact set $\varphi^{-1}(\overline{A.y_0})$, so $A \in \cM(G,X,x_0)$.

Because $\varphi$ necessarily is surjective, it follows that if one of the spaces $X$ and $Y$ is $\sigma$-compact, then so is the other. The last claim therefore follows from the the former together with Proposition~\ref{prop:co-cpt}.
\end{proof}

\noindent The two extreme values of the type are treated in the following two propositions. We omit the easy proof of the former.

\begin{proposition} \label{prop:compacttype}
Let $G$ be a countable group and let $(X,x_0)$ be a pointed locally compact $G$-space. The following conditions are equivalent:
\begin{enumerate}
\item $\cM(G,X,x_0) = P(G)$, \vspace{.1cm}
\item $X$ is co-compact and $\Type(G,X,x_0) = [G]$, \vspace{.1cm}
\item $X$ is compact.
\end{enumerate}
\end{proposition}

\begin{proposition} \label{prop:discrete}
Let $G$ be a countable group and let $(X,x_0)$ be a pointed locally compact $G$-space. Suppose that the isotropy group of $\{x_0\}$ is trivial, i.e., the map $g \mapsto g.x_0$ is injective. Then the following conditions are equivalent:
\begin{enumerate}
\item $\cM(G,X,x_0) = \cM_\mathrm{fin}$, \vspace{.1cm}
\item $X$ is co-compact and $\Type(G,X,x_0) = [\{e\}]$, \vspace{.1cm}
\item $X$ is discrete, \vspace{.1cm}
\item $(X,x_0)$ is isomorphic to $(G,e)$.
\end{enumerate}
\end{proposition}

\begin{proof}
(ii) $\Rightarrow$ (i) and  (iv) $\Rightarrow$ (iii) are trivial.

(i) $\Rightarrow$ (iv).  
Let $x \in X$ and find a relatively compact open neighbourhood $V$ of $x$. Then $A :=O_X(V,x_0)$ belongs to $\cM(G,X,x_0)$ (by Lemma \ref{lm:G-regular} (ii)), so $A$ is finite. Since $A.x_0 = G.x_0 \cap V$ is dense in $V$ (by Lemma \ref{lm:G-regular} (i)) and since finite sets are closed, we must have $A.x_0 = V$. This shows that $x \in V \subseteq G.x_0$. As $x \in X$ was arbitrary we conclude that $X = G.x_0$. The set $A.x_0$ is a finite and open, so each point in this set must be isolated. Hence every point in $X$ is isolated. The map $g \mapsto g.x_0$ is assumed to be injective, and it is therefore a $G$-isomorphism from $(G,e)$ onto $(X,x_0)$. 

(iii) $\Rightarrow$ (ii). Since $G.x_0$ is dense in $X$ and $X$ is discrete we have $X = G.x_0$, so $X$ is co-compact. The claim about the type follows from the assumption that the isotropy group of $x_0$ is trivial together with the fact that the compact subsets of a discrete set precisely are the finite subsets.
\end{proof}

\noindent We end this section with a result stating that there is a one-to-one correspondence between (not necessarily compact) left $G$-ideals in $P(G)$ and (not necessarily co-compact) open invariant subsets of $\beta G$, hence classifying the latter; see also Proposition~\ref{prop:X_A}. In particular, all left $G$-ideals arise from a pointed locally compact $G$-space. The result also extends Proposition \ref{prop:type(X_A)} to the non-compact case.

\begin{proposition} \label{prop:X_M}
For each left $G$-ideal $\cM$  put 
\begin{equation} \label{eq:X_M}
X_\cM = \bigcup_{A \in \cM} K_A \subseteq \beta G.
\end{equation}
Then $X_\cM$ is open, $G$-invariant, and $e \in X_\cM$. In particular, $(X_\cM,e)$ is a pointed locally compact $G$-space. Moreover,
\begin{enumerate}
\item $\cM(G,X_\cM,e) = \cM$. \vspace{.1cm}
\item Every open invariant subset of $\beta G$ is equal to $X_\cM$ for some left $G$-ideal $\cM$.
\end{enumerate}
\end{proposition}

\begin{proof} It is clear that $X_\cM$ is open. Invariance of $X_\cM$ follows from the fact that $g.K_A = K_{gA}$ (see \cite[Lemma 2.4]{KelMonRor:supramenable}) and the assumption that $\cM$ is a left $G$-ideal. Each non-empty open invariant subset of $\beta G$ contains $G = G.e$ as a dense orbit, so  $(X_\cM,e)$ is a pointed locally compact $G$-space.

(i). The proof follows the same idea as the proof of  Proposition \ref{prop:type(X_A)}. Let $A \subseteq G$. Then:
\begin{eqnarray*}
A \in \cM(G,X_\cM,e) &\overset{(1)}{\iff}& K_A \cap X_\cM \; \; \text{is compact} \\
&\overset{(2)}{\iff} & K_A \subseteq X_\cM \, \overset{(3)}{\iff} A \in \cM.
\end{eqnarray*}
(1) holds because the closure of $A=A.e$ in $X_\cM$ is equal to $K_A \cap X_\cM$; and (2) is established as in  the proof of Proposition \ref{prop:type(X_A)}. Let us look at (3): 
"$\Leftarrow$" holds by definition. Suppose that $K_A \subseteq X_\cM$. Then, by compactness of $K_A$, there exist $A_1, \dots, A_n \in \cM$ such that 
$$K_A \subseteq K_{A_1} \cup K_{A_2} \cup \cdots \cup K_{A_n} = K_{\bigcup_{i=1}^n A_i}.$$
Hence $A \subseteq \bigcup_{i=1}^n A_i$, and so $A \in \cM$.

(ii). Let $X$ be an open $G$-invariant subset of $\beta G$, and let $\cM$ be the set of all $A \subseteq G$ such that $K_A \subseteq X$. Then $\cM$ is a left $G$-ideal, and $X_\cM \subseteq X$ by \eqref{eq:X_M}. To prove the reverse inclusion, let $x \in X$. Since $\beta G$, and hence $X$, are totally disconnected  locally compact Hausdorff spaces, and since the compact-open sets in $\beta G$ are of the form $K_A$ for some $A \subseteq G$, there exists $A \subseteq G$ with $x \in K_A \subseteq X$. Hence $A \in \cM$, so  $x \in K_A \subseteq X_\cM$.
\end{proof}


\section{Non co-compact actions} \label{sec:cocompact}

\noindent A group $G$ is said to be \emph{supramenable} if it has no paradoxical subsets, cf.\ \cite{Rosenblatt:supramenable}. It was shown in \cite{KelMonRor:supramenable} that a group is supramenable if and only if any co-compact action of the group on a locally compact Hausdorff space admits a non-zero invariant Radon measure. We show here, using some of the machinery of the previous section, that the assumption that the action be co-compact cannot be removed. In fact, any infinite countable group admits an action on a  locally compact Hausdorff space which does not admit a non-zero invariant Radon measure (Proposition~\ref{prop:non-co-cpt}). We also prove the existence of non-compact left $G$-ideals and non co-compact open $G$-invariant subsets of $\beta G$ for every infinite countable group (Proposition~\ref{prop:non-co-compact}).

We start by giving an obstruction to having a non-zero invariant Radon measure:

\begin{lemma} \label{lm:co-cpt-obstruction}
Let $G$ be a countable group acting on a locally compact Hausdorff space $X$. Suppose that there is a sequence $\{K_n\}_{n=1}^\infty$ of compact subsets of $X$ such that:
\begin{enumerate}
\item For each $n$ there is an infinite subset $J_n$ of $G$ such that $\{g.K_n\}_{g \in J_n}$ are pairwise disjoint subsets of $K_{n+1}$,
\item $\displaystyle{X= \bigcup_{n=1}^\infty \bigcup_{g \in G} g.K_n.}$
\end{enumerate}
Then $X$ admits no non-zero invariant Radon measure. 
\end{lemma}

\begin{proof} Suppose that $\lambda$ is an invariant Radon measure on $X$. For each $n \in \N$ we have that $\lambda(K_{n+1}) < \infty$, which by (i) and invariance of $\lambda$ entails that $\lambda(K_n) = 0$. It then follows from (ii) and invariance that $\lambda(X)=0$.
\end{proof}

\noindent We thank the referee for suggesting the proof of the following lemma (which allowed us to remove the condition that $G$ contains an element of infinite order). 

\begin{lemma}  \label{lm:I_n}
Let $G$ be an infinite countable group. 
Then there is a sequence $I_1,I_2,I_3,\dots$ of infinite subsets of $G$ 
such that the product maps 
\[
\Phi_n\colon I_n\times\cdots\times I_2\times I_1\to G,\quad 
\Phi_n(g_n,\dots,g_2,g_1)=g_n\cdots g_2g_1,\quad g_j\in I_j, 
\]
are injective for each $n\in\N$, 
and such that $e\in I_n$ for all $n\in\N$. 
\end{lemma}

\begin{proof}
If $G$ has an element $g$ of infinite order, then the proof is very easy: Choose mutually distinct natural numbers $\{k(n,j)\}$ for all integers $n \ge 1$ and $j \ge 2$, put $m(n,j) = 2^{k(n,j)}$, and put
$$I_n = \{e,g^{m(n,2)}, g^{m(n,3)}, g^{m(n,4)},\cdots\}.$$
Then $I_1,I_2, I_3, \dots$  have the desired properties.

Consider now the general case, and denote the center of the group $G$ by $Z(G)$. 
For $g\in G$, we denote the centralizer of $g$ in $G$ by $C_G(g)$. 
First we prove the lemma in the case 
where $C_G(g)$ is finite for every $g\in G\setminus Z(G)$ 
(this implies either $Z(G)$ is finite or $G=Z(G)$, i.e., $G$ is abelian). 
It is easy to see that
for any $h\in G$ and any finite subset $F\subset G$, 
the set of $g\in G$ such that 
$\{ghg^{-1}\}\cap F$ is not contained in $Z(G)$ is finite. 
Hence for any finite subsets $F_0,F_1\subset G$, 
the set of $g\in G$ such that 
$gF_0g^{-1}\cap F_1\subset Z(G)$ is infinite. 

We will define each of the sets $I_n$ 
as an increasing union $I_n=\bigcup_{j=0}^\infty I_n^j$, 
where each $I_n^j$ is finite, and $I_n^0=\{e\}$. 
Fix a sequence $\{n_j\}$ of natural numbers with $n_j\leq j$, 
in which every natural number appears infinitely often. 
At stage $j$ of the construction 
we will add a single element to $I_{n_j}^{j-1}$ to obtain $I_{n_j}^j$ 
(and take $I_n^j=I_n^{j-1}$ for $n\neq n_j$). 

The sets $I_n^j$ are defined inductively as follows. For $j=1$ we have $n_j=1$. Take any $g_1 \ne e$, set $I_1^1= I_{0}^1 \cup \{g_1\}$, and set $I_n^1 = I_n^0$ for $n \ne 1$. Let now $j \ge 2$ be given and assume that the sets $I_n^{j-1}$, $n\in\N$, are defined such
that the restriction of $\Phi_{j-1}$ 
to $I_{j-1}^{j-1}\times\dots\times I_2^{j-1}\times I_1^{j-1}$ is injective. 
For $n\neq n_j$, define $I_n^j=I_n^{j-1}$. 
Let 
$$F_0=I_{n_j-1}^jI_{n_j-2}^j\dots I_1^j, \qquad F_1=I_j^jI_{j-1}^j\dots I_{n_j+1}^j.$$
Since $F_0F_0^{-1}$ and $F_1^{-1}F_1$ are finite, 
the assumption on $G$ implies that 
the set $A$, of all group elements $g\in G$ such that 
$gF_0F_0^{-1}g^{-1}\cap F_1^{-1}F_1\subseteq Z(G)$, is infinite. 
Let $g_j$ be any element of $A\setminus F_1^{-1}F_1I_{n_j}^{j-1}F_0F_0^{-1}$ 
and define $I_{n_j}^j=I_{n_j}^{j-1}\cup\{g_j\}$. 
We claim that the restriction of $\Phi_j$ 
to $I_j^j\times\dots\times I_2^j\times I_1^j$ is injective. 
By induction (and since $I_j^{j-1}=\{e\}$) 
the restriction of $\Phi_j$ 
to $I_j^{j-1}\times\dots\times I_2^{j-1}\times I_1^{j-1}$ is injective. 
The choice of $g_j\notin F_1^{-1}F_1I_{n_j}^{j-1}F_0F_0^{-1}$ ensures that 
\[
F_1g_jF_0\cap F_1I_{n_j}^{j-1}F_0=\emptyset, 
\]
so it remains to show that the restriction of $\Phi_j$ to the set 
\[
I_j^j\times I_{j-1}^j\times\dots\times I_{n_j+1}^j\times\{g_j\}
\times I_{n_j-1}^j\times I_{n_j-2}^j\times\dots\times I_1^j
\]
is injective. 
Suppose that 
\[
h_j\dots h_{n_j+1}g_jh_{n_j-1}\dots h_1
=k_j\dots k_{n_j+1}g_jk_{n_j-1}\dots k_1, 
\]
where $h_i,k_i\in I_i^j$. 
Then 
\[
g_jk_{n_j-1}\dots k_1(h_{n_j-1}\dots h_1)^{-1}g_j^{-1}
=(k_j\dots k_{n_j+1})^{-1}h_j\dots h_{n_j+1},
\]
which is in the set $g_jF_0F_0^{-1}g_j^{-1}\cap F_1^{-1}F_1\subseteq Z(G)$. 
Thus 
\[
k_j\dots k_{n_j+1}k_{n_j-1}\dots k_1
=h_j\dots h_{n_j+1}h_{n_j-1}\dots h_1
\]
so by the induction hypothesis, $h_i=k_i$ for all $i$, 
which finishes the proof in this case. 

If there exists an infinite subgroup of $G$ 
which satisfies the hypothesis above, 
then we can apply the argument above to this subgroup. 
So, we may assume that $G$ does not have such a subgroup. 
Then there exists $g_1\in G\setminus Z(G)$ such that $C_G(g_1)$ is infinite. 
By considering the subgroup $C_G(g_1)$, 
we can find $g_2\in C_G(g_1)\setminus Z(C_G(g_1))$ 
such that $C_{C_G(g_1)}(g_2)=C_G(g_1)\cap C_G(g_2)$ is infinite. 
In the same way, 
there exists $g_3\in(C_G(g_1)\cap C_G(g_2))\setminus Z(C_G(g_1)\cap C_G(g_2))$ 
such that $C_G(g_1)\cap C_G(g_2)\cap C_G(g_3)$ is infinite. 
Repeating this argument, we obtain elements $g_1,g_2,g_3,\dots$, 
which generate an infinite abelian subgroup. 
This is a contradiction. 
\end{proof}

\begin{proposition} \label{prop:non-co-cpt}
Let $G$ be an infinite countable  group. Then there is a locally compact $\sigma$-compact Hausdorff space $X$ on which $G$ acts freely and with a dense orbit,  such that there is no non-zero invariant Radon measure on $X$. If $G$ is supramenable, then the action is necessarily non co-compact. 

The space $X$ can be chosen to be an open invariant subset of $\beta G$. 
\end{proposition}

\begin{proof} Let $\{I_n\}_{n \ge 1}$ be as in Lemma \ref{lm:I_n}, and put 
$$A_n = I_n I_{n-1} \cdots I_1 = \Phi_n(I_n,I_{n-1}, \dots, I_1)$$
for each $n \ge 1$. Then $\{A_n\}_{n=1}^\infty$ is increasing, because $e \in I_n$ for all $n\ge 1$. Put 
$$\cM = \bigcup_{n=1}^\infty \cM_{A_n},$$
and put $X = X_\cM$, cf.\ \eqref{eq:X_M}. Then $X$ is an open invariant subset of $\beta G$. We show that conditions (i) and (ii) of Lemma \ref{lm:co-cpt-obstruction} are satisfied with $K_n = K_{A_n}$.

It follows from injectivity of $\Phi_{n+1}$ that the sequence of sets $\{gA_n\}_{g \in I_{n+1}}$ are pairwise disjoint. By construction, $gA_n \subseteq A_{n+1}$ for all $g \in I_{n+1}$. As $K_{gA} = g.K_{A}$ for all $g \in G$ and all $A \subseteq G$, and since $K_A \cap K_B = \emptyset$ whenever $A \cap B = \emptyset$, we conclude that $\{g.K_{A_n}\}_{g \in I_{n+1}}$ is a sequence of pairwise disjoint subsets of $K_{A_{n+1}}$. Hence  Lemma \ref{lm:co-cpt-obstruction} (i) holds. 

By \eqref{eq:X_M},
$$X=X_\cM =  \bigcup_{A \in \cM} K_A.$$
Let $A \in \cM$ be given. Then $A \in \cM_{A_n}$ for some $n \in \N$, which means that $A \propto A_n$, so $A \subseteq \bigcup_{g \in F} gA_n$ for some finite subset $F$ of $G$. This entails that 
$$K_A \subseteq \bigcup_{g \in F} g.K_{A_n}.$$
As $A \in \cM$ was arbitrary, we see that condition (ii) in  Lemma \ref{lm:co-cpt-obstruction}  holds. Hence there is no non-zero invariant Radon measure on $X$. Since condition (ii) in  Lemma \ref{lm:co-cpt-obstruction} holds, $X$ is $\sigma$-compact. Each non-empty open $G$-invariant subset of $\beta G$ contains the dense set $G = G.e$, so $X$ contains a dense orbit.

If $G$ is supramenable, then the action must be non co-compact, by
\cite[Theorem 1]{KelMonRor:supramenable}.
\end{proof}

\noindent For completeness we show that non-compact left $G$-ideals and locally compact $G$-spaces, that are not co-compact, exist for every infinite countable group:

\begin{proposition} \label{prop:non-co-compact}
In each  infinite countable group $G$ there exists a left $G$-ideal that is not compact; and there exists a locally compact $\sigma$-compact Hausdorff space $X$ on which $G$ acts freely with a dense orbit, but not co-compactly.
\end{proposition}

\begin{proof} We construct an increasing sequence of subsets $\{A_n\}_{n=1}^\infty$ of $G$ such that $A_{n+1} \npropto A_n$ for all $n$. It will then follow that 
$$\cM = \bigcup_{n=1}^\infty \cM_{A_n}$$
is non-compact. Indeed, if it were compact, then $\cM = \cM_{A_n}$ for some $n$, which would entail that $A_{n+1} \propto A_n$ contrary to the construction. The $G$-space $X=X_\cM$ will therefore be non co-compact by Proposition~\ref{prop:X_M} and Proposition~\ref{prop:co-cpt}, and $G$ acts freely and with a dense orbit on $X$. 

To construct the sets $A_1, A_2, A_3,\dots$  take a proper right-invariant metric $d$ on $G$. Choose inductively a sequence $\{g_n\}_{n=1}^\infty$ of elements in $G$ such that $d(g_{n}, \{g_1, \dots, g_{n-1}\}) \ge n$ for $n \ge 2$. Let each $A_n$ be a subset of $\{g_1,g_2,g_3, \dots\}$ such that $\{A_n\}$ is increasing and  $A_{n+1} \setminus A_n$ is infinite for all $n$. It then follows from Lemma \ref{lm:approx-geom} that $A_{n+1} \npropto A_n$ for all $n$. 
\end{proof}


\section{A universal property of the $G$-space $X_A$} \label{sec:X_A}

\noindent The $G$-space $\beta G$ of a (discrete) group $G$ is universal among all compact $G$-spaces with a dense orbit in the following sense: for each compact $G$-space $X$ with a dense orbit $G.x_0$ there is a surjective continuous $G$-map $\varphi \colon \beta G \to X$ (which is unique if we also require that $\varphi(e) = x_0$). 

We show in this section that the pointed locally compact $G$-space $(X_A,e)$ (associated with a non-empty subset $A \subseteq G$)  is universal among pointed locally compact co-compact $G$-spaces of type $[A]$, see also Proposition \ref{prop:sametype}. The  morphisms between locally compact $G$-spaces  are \emph{proper} continuous $G$-maps.

For each $f \in C(\beta G)$ denote by $\hat{f} \in \ell^\infty(G)$ the restriction of $f$ to $G$.

\begin{lemma} \label{lm:C_A}
Let $A$ be a non-empty subset of a group $G$ and let $f \in C(\beta G)$. It follows that $f \in C_0(X_A)$ if and only if for each $\ep >0$ there exists $B \subseteq G$ with $B \propto A$ such that 
$$g \in G \setminus B \implies |\hat{f}(g)| < \ep.$$
\end{lemma}

\begin{proof} Suppose first that $f \in C_0(X_A)$ and let $\ep > 0$ be given. Then there is a compact subset $L$ of $X_A$ such that $|f(x)| < \ep$ for all $x \in X_A \setminus L$. Take a finite subset $F$ of $G$ such that $L \subseteq \bigcup_{g \in F} g.K_A$ and put $B = FA \propto A$. If $g \in G \setminus B$, then $g \notin L$, so $|\hat{f}(g)| = |f(g)|   <\ep$. 

To prove the "if"-part, let $f \in C(\beta G)$ and suppose that $f$ has the stipulated property. Let $\ep >0$, and let $B \subseteq G$ with $B \propto A$ be such that $|\hat{f}(g)| < \ep$ for all $g \in G \setminus B$. Then $K_B$ is a compact-open subset of $X_A$ and if $g \in G \cap (X_A\setminus K_B)$, then $g \notin B$, so $|f(g)| = |\hat{f}(g)|  < \ep$. As $G \cap (X_A \setminus K_B)$ is dense in $X_A \setminus K_B$, we conclude that $|f(x)| \le \ep$ for all $x \in X_A \setminus K_B$. This proves that $f \in C_0(X_A)$. 
\end{proof}

\noindent The theorem below says that the pointed locally compact $G$-space $(X_A,e)$  is universal for the class of pointed locally compact co-compact $G$-spaces of type $[A]$.

\begin{theorem} \label{thm:universal}
Let $A$ be a non-empty subset of a countable group $G$. Then for each pointed locally compact  $G$-space $(X,x_0)$, there is a (necessarily unique and surjective) proper continuous $G$-map $\varphi \colon X_A \to X$ with $\varphi(e) = x_0$ if and only if $X$ is co-compact and $\Type(G,X,x_0) = [A]$.
\end{theorem}

\begin{proof} The "only if" part follows from Propositions \ref{prop:sametype} and \ref{prop:type(X_A)}, and the fact that $X$ is co-compact if it is the image of the co-compact space $X_A$ under a continuous $G$-map.

Suppose that $(X,x_0)$ is a pointed locally compact co-compact $G$-space of type $[A]$. Define a homomorphism $\hat{\pi} \colon C_0(X) \to \ell^\infty(G)$ by
$$\hat{\pi}(f)(g) = f(g.x_0), \qquad f \in C_0(X), \quad g \in G.$$
Composint $\pi$ with the canonical isomorphism $\ell^\infty(G) \to C(\beta G)$ we obtain a homomorphism $\pi \colon C_0(X) \to C(\beta G)$ which satisfies $\widehat{\pi(f)}=\hat{\pi}(f)$ for all $f \in C_0(X)$. The homomorphisms $\pi$ and $\hat{\pi}$ are injective because $G.x_0$ is dense in $X$.

We claim that $\pi(f) \in C_0(X_A)$ for all $f \in C_0(X)$. 
Let $\ep > 0$, and let $L$ be a compact subset of $X$ such that $|f(x)| < \ep$ for all $x \in X \setminus L$. Put $B = O_X(L,x_0)$. Then $B \propto A$ by Proposition \ref{prop:regular}. If $g \in G \setminus B$, then $g.x_0 \notin L$, so $|\hat{\pi}(f)(g)| = |f(g.x_0)| < \ep$. It therefore follows from Lemma~\ref{lm:C_A} that $\pi(f) \in C_0(X_A)$.

We have now obtained a homomorphism $\pi \colon C_0(X) \to C_0(X_A)$ satisfying $\pi(f)(g) = f(g.x_0)$. Since $G$ is dense in $X_A$ we see that $\pi$ is $G$-equivariant. We show that the image of $\pi$ is full in $C_0(X_A)$. To  this end we must show that for each $y \in X_A$ there exists $f \in C_0(X)$ such that $\pi(f)(y) \ne 0$. Take $g \in G$ such that $y \in g.K_A = K_{gA}$. Put $K = \overline{A.x_0}$, which is a compact subset of $X$ by Proposition \ref{prop:regular}. 
Let $f \in C_0(X)$ be such that the restriction of $f$ to $g.K$ is $1$. For each $h \in gA$ we have $h.x_0 \in g.K$, so $\pi(f)(h) = f(h.x_0) = 1$. Hence $\pi(f)(x) = 1$ for all $x$ in the closure of $gA$ in $X_A$, and this closure is precisely $K_{gA}$. This entails that $\pi(f)(y) = 1$.

In conclusion we obtain a continuous proper $G$-equivariant epimorphism $\varphi \colon X_A \to X$ such that $\pi(f) = f \circ \varphi$. In particular, $\varphi(e) = x_0$. 
\end{proof}

\noindent 
It was shown in  \cite[Proposition 2.6]{KelMonRor:supramenable} that there is a non-zero invariant Radon measure on the $G$-space $X_A$ if and only if $A$ is non-paradoxical. Using the theorem above we can extend this result as follows:

\begin{corollary} \label{cor:Radon}
Let $G$ be a group and let $A$ be a non-empty subset of $G$. Then $A$ is non-paradoxical if and only if every pointed locally compact co-compact $G$-space $(X,x_0)$ of type $[A]$ admits a non-zero invariant Radon measure.
\end{corollary}

\begin{proof} The "if" part follows from  \cite[Proposition 2.6]{KelMonRor:supramenable} with $(X,x_0) = (X_A,e)$. Suppose that $A$ is non-paradoxical. 
Use Theorem~\ref{thm:universal}  to find a proper continuous $G$-map $\varphi \colon X_A \to X$ (such that $\varphi(e) = x_0$). Let $\mu$ be a non-zero invariant Radon measure on $X_A$, cf.\ \cite[Proposition 2.6]{KelMonRor:supramenable}, and let $\tilde{\mu} = \mu \circ \varphi^{-1}$ be the image measure on $X$. Then $\tilde{\mu}$ is a non-zero invariant Radon measure on $X$.
\end{proof}

\noindent The existence of a non-zero invariant Radon measure is of course independent on the choice of base point, and, indeed, a "base point free" version of this results holds, see Proposition~\ref{prop:paradoxical}. 

One can alternatively obtain Corollary~\ref{cor:Radon} from 
\cite[Lemma~2.2]{KelMonRor:supramenable} in combination with Proposition~\ref{prop:type-condition} without making reference to Theorem~\ref{thm:universal}.


\section{Universal locally compact minimal $G$-spaces} \label{sec:min-universal}

\noindent 
In the previous section we showed that the  co-compact open $G$-invariant subspaces of $\beta G$, which we know are of the form $X_A$ for some $A \subseteq G$, are universal among a certain class of locally compact $G$-spaces. In this section we shall show that the minimal closed $G$-invariant subsets of $X_A$ are universal among certain minimal locally compact $G$-spaces (provided $A$ is of \emph{minimal type}), and that any two  minimal closed $G$-invariant subsets of $X_A$ are isomorphic as $G$-spaces. 

Ellis proved in \cite{Ellis:minimal} that the minimal closed invariant subsets of the $G$-space $\beta G$  are universal and pairwise isomorphic.  Y.\ Gutman and H.\ Li gave in \cite{GutLi:universal} a short and elegant new proof of this theorem. We shall mimic their proof in our proof of Theorem~\ref{thm:unique-minimal} below.

\begin{definition} \label{def:minimal-type}
For a group $G$, let $P_\approx^\mathrm{min}(G)$ be the set of equivalence classes $[A] \in P_\approx(G)$ that arise as the type of a minimal pointed locally compact $G$-space. A subset $A \subseteq G$ for which $[A] \in P_\approx^{\mathrm{min}}(G)$  is said to be of minimal type. 
\end{definition}

\noindent In other words, a subset $A$ of $G$ is of minimal type if and only if there is a pointed locally compact minimal $G$-space $(X,x_0)$ such that $\Type(G,X,x_0) = [A]$. Recall that minimal locally compact $G$-spaces automatically are co-compact.

\begin{example} \label{ex:subgroup-2} Each subgroup $H$ of a group $G$ is of minimal type.  This can be seen as follows: If $H$ is finite, then $[H]$ is the type of the trivial (minimal) $G$-space $(G,e)$, cf.\ Proposition \ref{prop:discrete}. 

Suppose next that $H$ is an infinite subgroup of $G$. Then $H$ admits a free minimal action on a compact Hausdorff space $X$, see for example \cite{Ellis:minimal}. This induces a free minimal action of $G$ on $Y =X \times (G/H)$, where $G/H$ is the left coset, see \cite[Lemma 7.1]{KelMonRor:supramenable}. The set $K = X \times \{e\}$ is a $G$-regular compact subset of $Y$.  Choose $x_0 \in X$, and put $y_0 = (x_0,e)$. Then $g.y_0 \in K$ if and only if $g \in H$, so $O_Y(K,y_0) = H$. It follows from Proposition \ref{prop:type-condition} that $\Type(G,Y,y_0) = [H]$, so $[H]$ is of minimal type. 

In particular, $[G]$ and $[\{e\}]$ are of minimal types representing compact minimal $G$-spaces, which always exist, respectively, the trivial discrete $G$-space $G$. 
\end{example}

\noindent  A priori, it is not clear if every infinite group $G$ contains minimal types other than $[G]$ and $[\{e\}]$. In other words, does every infinite group $G$ act minimally on some locally compact non-compact and non-discrete space. We shall answer this question affirmatively in Section \ref{sec:non-discrete-compact}.  

In Section \ref{sec:basepoint} we shall give an intrinsic description of the subsets of a group that are of minimal type. Far from all subsets of a group are of minimal type. If $A \subseteq G$ is of minimal type, then we shall show that the minimal closed $G$-invariant subsets of $X_A$ have the following nice universal property:

\begin{definition} \label{def:min-universal} Let $G$ be a group and let $A$ be a non-empty subset of $G$ of minimal type. 

A locally compact minimal $G$-space $Z$ is said to be a \emph{universal minimal $G$-space of type $A$} if for each pointed minimal locally compact $G$-space $(X,x_0)$ with $\Type(G,X,x_0) = [A]$ there is a surjective proper continuous $G$-map $\varphi \colon Z \to X$.
\end{definition}

\begin{lemma} \label{lm:type-of-universal}
Let $G$ be a group and let $A$ be a non-empty subset of $G$ of minimal type. Suppose that $Z$ is a  universal minimal $G$-space of type $A$. Then $\Type(G,Z,z_0) = [A]$ for some $z_0 \in Z$. 
\end{lemma}

\begin{proof} Since $A$ is of minimal type there is a pointed locally compact $G$-space $(X,x_0)$ with $\Type(G,X,x_0) = [A]$. By the universal property of $Z$ there is a surjective proper continuous $G$-map $\varphi \colon Z \to X$. Choose $z_0 \in Z$ such that $\varphi(z_0) = x_0$. It then follows from Proposition \ref{prop:sametype} that $\Type(G,Z,z_0) =\Type(G,X,x_0) = [A]$.
\end{proof}

\begin{proposition} \label{prop:universal_minimal}
Let $G$ be a group and let $A \subseteq G$ be of minimal type. Then any  closed minimal $G$-invariant subset $Z$ of $X_A$ is a universal minimal $G$-space of type $A$. 
\end{proposition}

\begin{proof} 
Let $(X,x_0)$ be any pointed minimal locally compact $G$-space of type $[A]$. By Theorem \ref{thm:universal} there is a continuous proper $G$-epimorphism $\varphi \colon X_A \to X$ with $\varphi(e) = x_0$. Since a proper continuous map between locally compact Hausdorff spaces maps closed sets to closed sets, it follows that  $\varphi(Z)$ is a closed $G$-invariant subset of $X$. Hence $\varphi(Z) = X$, so the restriction of $\varphi$ to $Z$ has the desired properties.
\end{proof}

\noindent
We proceed to prove that universal minimal $G$-spaces of a given type are essentially unique. First we need to know that we can take projective limits in the class of pointed locally compact $G$-spaces of a given type:

\begin{lemma} \label{lm:proj-limit}
Let $G$ be a group and let $A\subseteq G$ be of minimal type. Let $I$ be an upwards directed totally ordered set, let $(X_\alpha,x_\alpha)_{\alpha \in I}$ be a family of pointed minimal locally compact $G$-spaces of type $[A]$ equipped with surjective proper continuous $G$-maps $\varphi_{\beta,\alpha} \colon X_\alpha \to X_\beta$, for $\alpha > \beta$, satisfying 
\begin{enumerate}
\item $\varphi_{\beta,\alpha}(x_\alpha) = x_\beta$ for all $\alpha > \beta$, \vspace{.1cm}
\item $\varphi_{\gamma,\beta} \circ \varphi_{\beta,\alpha} = \varphi_{\gamma, \alpha}$ for all $\alpha > \beta > \gamma$.
\end{enumerate}
Then there is a pointed minimal locally compact $G$-space $(X,x_0)$ of type $[A]$ and surjective proper continuous $G$-maps $\varphi_{\alpha} \colon X \to X_\alpha$ satisfying $\varphi_{\alpha}(x_0) = x_\alpha$ for all $\alpha \in I$, and $\varphi_{\beta,\alpha} \circ \varphi_\alpha = \varphi_\beta$ for all $\alpha > \beta$. 
\end{lemma}

\noindent The $G$-space $(X,x_0)$ with the mappings $\varphi_\alpha \colon X \to X_\alpha$ is the \emph{projective limit} of the family $(X_\alpha,x_\alpha)_{\alpha \in I}$. 

\begin{proof} Put
\[
X=\Big\{(z_\alpha)_{\alpha \in I} \in\prod_{\alpha \in I }X_\alpha \;
\big| \; \varphi_{\gamma,\beta}(z_\beta)=z_\gamma \; \;  \text{for all} \;
\gamma<\beta\Big\},
\]
put $x_0 = (x_\alpha)_{\alpha \in I} \in X$, and let $\varphi_\alpha \colon X \to X_\alpha$ be the (restriction to $X$ of the) projection map.  Equip the space $\prod_{\alpha \in I } X_\alpha$ with the product topology, and equip $X \subseteq \prod_{\alpha \in I } X_\alpha$ with the subspace topology. Note that the product space is not locally compact if $I$ is an infinite set and  the spaces $X_\alpha$ are non-compact.

The projection maps $\varphi_\alpha$ are continuous, also when restricted to $X$. The maps $\varphi_\alpha$ are also proper (when restricted to $X$). Indeed, if $K \subseteq X_\alpha$ is compact, then $\varphi_\alpha^{-1}(K)$ is a closed (and hence compact) subset of the compact set $\prod_{\beta \in I} K_\beta$, where
$$K_\beta = \begin{cases} K, & \beta = \alpha, \\ \varphi_{\beta,\alpha}(K), & \beta < \alpha, \\ \varphi_{\alpha,\beta}^{-1}(K), & \beta > \alpha.\end{cases}$$

It is straightforward to show that the family of sets of the form $\varphi_\alpha^{-1}(U)$, where $\alpha \in I$ and $U \subseteq X_\alpha$ is open, is closed under intersections, and hence forms a basis for the topology on $X$. Using this, and the fact, established above, that each $\varphi_\alpha$ is proper, we conclude that $X$ is locally compact (and Hausdorff). 

Equip $X$ with the natural $G$-action. Then each $\varphi_\alpha$ becomes a continuous, proper $G$-map. We show that each orbit in $X$ is dense.  Take $z \in X$ and take a non-empty  open subset $V$ of $X$ of the form $\varphi_\alpha^{-1}(U)$, where $\alpha \in I$ and where $U \subseteq X_\alpha$ is open, cf.\ the comment above. We must show that $G.z \cap V \ne \emptyset$. Write $z = (z_\alpha)$ and use that $X_\alpha$ is a minimal $G$-space to find $g\in G$ with $g.z_\alpha \in U$. Then $\varphi_\alpha(g.z) = g.z_\alpha \in U$, so $g.z \in \varphi_{\alpha}^{-1}(U)=V$ as desired. 

It finally follows from Proposition \ref{prop:sametype} that 
$$\Type(G,X,x_0) = \Type(G,X_\alpha,x_\alpha) = [A].$$
\end{proof}

\noindent The proof of the next theorem closely follows the ideas from \cite{GutLi:universal}: 

\begin{theorem} \label{thm:unique-minimal}
Let $G$ be a group and let $A \subseteq G$ be a subset of minimal type. Let $Z_1$ and $Z_2$ be universal minimal $G$-spaces of type $A$. Then $Z_1$ and $Z_2$ are isomorphic as $G$-spaces. 
\end{theorem}

\begin{proof} 
It follows from Lemma~\ref{lm:type-of-universal} that there exist points $z_i\in Z_i$ such that $\Type(G,Z_i,z_i) = [A]$ for $i=1,2$. Since $Z_1$ and $Z_2$ are universal spaces of type $A$ we can find surjective proper continuous $G$-maps 
$\varphi_1 \colon Z_1\to Z_2$ and $\varphi_2 \colon Z_2\to Z_1$. 
(We do not expect $\varphi_1(z_1)=z_2$ and $\varphi_2(z_2)=z_1$.) 
Thus $\varphi_2\circ\varphi_1 \colon Z_1\to Z_1$  and $\varphi_1\circ\varphi_2 \colon Z_2\to Z_2$
are surjective proper continuous $G$-maps. 

It therefore suffices to show that if $Z$ is a universal minimal $G$-space of type $A$,  then every surjective continuous proper $G$-map $\rho \colon Z \to Z$ is injective. Suppose that $\rho$ is not injective. We apply the argument from \cite{GutLi:universal}  to reach a contradiction.

Let $\alpha$ be an ordinal with $\lvert\alpha\rvert>\lvert Z^2\rvert$. 
We shall construct
a pointed locally compact $G$-space $(X_\beta,x_\beta)$ of type $[A]$ for each ordinal $\beta \le \alpha$ and 
proper continuous $G$-maps 
$\psi_{\gamma,\beta} \colon X_\beta\to X_\gamma$ for $\gamma<\beta\leq\alpha$ 
with the following properties: 
\begin{itemize}
\item $\psi_{\gamma,\beta}(x_\beta)=x_\gamma$. 
\item $\psi_{\delta,\gamma}\circ\psi_{\gamma,\beta}=\psi_{\delta,\beta}$ 
for any $\delta<\gamma<\beta \le \alpha$. 
\item For each ordinal $\beta<\alpha$, 
there exist distinct elements $y_\beta,z_\beta\in X_{\beta+1}$ 
with $\psi_{\beta,\beta+1}(y_\beta)=\psi_{\beta,\beta+1}(z_\beta)$. 
\end{itemize}
Suppose the above $G$-spaces and $G$-maps have been constructed. 
By the universal property of $Z$ one has an epimorphism $Z\to X_\alpha$. 
This implies that $\lvert X_\alpha\rvert\leq\lvert Z\rvert$. Let $\beta < \alpha$ be given, and consider the surjective $G$-maps:
$$\xymatrix@C+1.5pc{X_\beta & X_{\beta+1} \ar[l]_-{\psi_{\beta,\beta+1}} & X_\alpha \ar[l]_-{\psi_{\beta+1,\alpha}}}$$
Find  $\tilde y_{\beta},\tilde z_{\beta}\in X_\alpha$ with $\psi_{\beta+1,\alpha}(\tilde y_{\beta})=y_\beta$ and $\psi_{\beta+1,\alpha}(\tilde z_{\beta})=z_\beta$.

For any $\gamma<\beta<\alpha$, one has
\begin{eqnarray*}
\psi_{\gamma+1,\alpha}(\tilde y_\beta) &=& (\psi_{\gamma+1,\beta} \circ \psi_{\beta,\beta+1})(y_\beta) \, = \,  (\psi_{\gamma+1,\beta} \circ \psi_{\beta,\beta+1})(z_\beta) \\
&=& \psi_{\gamma+1,\alpha}(\tilde z_\beta)
\end{eqnarray*}
whereas $\psi_{\gamma+1,\alpha}(\tilde y_\gamma)
=y_\gamma\neq z_\gamma=\psi_{\gamma+1,\alpha}(\tilde z_\gamma)$.
Hence 
$(\tilde y_\beta,\tilde z_\beta)\neq(\tilde y_\gamma,\tilde z_\gamma)$ whenever $\beta \ne \gamma$, so that 
the map $$\{\beta\mid0\leq\beta<\alpha\}\to X_\alpha\times X_\alpha, \qquad \beta\mapsto(\tilde y_\beta,\tilde z_\beta),$$ is injective. 
This implies that $\lvert\alpha\rvert\leq\lvert X_\alpha^2\rvert$, which leads to the contradiction  $\lvert Z^2\rvert<\lvert\alpha\rvert
\leq\lvert X_\alpha^2\rvert\leq\lvert Z^2\rvert$. 

The construction of the $G$-spaces and the $G$-maps above is carried out through transfinite induction. For $\beta=0$, let $X_0 = Z$, and choose $x_0 \in X_0$ such that  the type of the pointed locally compact $G$-space $(X_0,x_0)$ is $[A]$, cf.\ Lemma~\ref{lm:type-of-universal}.

If $\beta$ is a successor ordinal,  set $X_\beta = Z$ and use the universal property of  the space $Z$ to find a surjective continuous proper $G$-map $\varphi_\beta \colon X_\beta \to X_{\beta-1}$. Set
$\psi_{\beta-1,\beta}=\varphi_\beta\circ\rho$, and 
for $\gamma<\beta-1$, 
set $\psi_{\gamma,\beta}=\psi_{\gamma,\beta-1}\circ\psi_{\beta-1,\beta}$. 
Let $x_\beta\in X_\beta$ be 
a preimage of $x_{\beta-1}$ by $\psi_{\beta-1,\beta}$. 
Then $(X_\beta,x_\beta)$ is of type $[A]$ by Proposition \ref{prop:sametype}. Since $\rho$ is not injective we can find distinct elements $y_\beta, z_\beta \in X_\beta$ such that $\psi_{\beta,\beta+1}(y_\beta) = \psi_{\beta,\beta+1}(z_\beta)$.

If $\beta$ is a limit ordinal, then 
define $(X_\beta,x_\beta)$ to be the projective limit of 
$(X_\gamma)_{\gamma<\beta}$, cf.\ Lemma \ref{lm:proj-limit};
and for $\gamma<\beta$, define $\psi_{\gamma,\beta} \colon X_\beta\to X_\gamma$ 
to be the epimorphism coming from the projective limit. 
\end{proof}

\begin{corollary} \label{cor:isomorphism}
Let $G$ be a group and let $A$ be a subset of $G$ of minimal type. Then any two minimal closed $G$-invariant subsets of $X_A$ are isomorphic as $G$-spaces.
\end{corollary}

\begin{proof} This follows immediately from Proposition~\ref{prop:universal_minimal} and Theorem~\ref{thm:unique-minimal}.
\end{proof}


\section{Minimal closed invariant subsets of $X_A$}

\noindent Fix a group $G$ and a non-empty subset $A$ of $G$, and let $X_A$ be the open $G$-invariant subset of $\beta G$ associated with $A$. Consider the following questions:

\begin{itemize}
\item Are any two minimal closed $G$-invariant subsets of $X_A$ isomorphic as $G$-spaces?
\item For which $A$ is it possible to find a minimal closed $G$-invariant subset of $X_A$ which is compact, respectively, discrete, respectively, neither compact nor discrete?
\end{itemize}

\noindent The answer to first question is affirmative when $A$ is of minimal type (Corollary~\ref{cor:isomorphism}). In general this question has negative answer (see Example~\ref{ex:four}, Case III), two minimal closed $G$-invariant subsets of $X_A$ need not even be homeomorphic. 

The second question, in the case of compact minimal subsets, was answered in \cite{KelMonRor:supramenable}:

\begin{definition}[{cf.\  \cite[Definition 3.5]{KelMonRor:supramenable}}] \label{def:absorbing} 
A non-empty subset $A$ of a group $G$ is said to be \emph{absorbing} if all finite subset $F$ of $G$ there exists $g \in G$ such that $Fg \subseteq A$, 
\end{definition}

\noindent Equivalently, $A$ is absorbing if $\bigcap_{g \in F} gA \ne \emptyset$ for all finite subsets $F$ of $G$. 

\begin{proposition}[{cf.\ \cite[Proposition 5.5 (iv)]{KelMonRor:supramenable}}] \label{prop:absorbing}
Let $A$ be a non-empty subset of a group $G$. Then $X_A$ contains a compact minimal closed $G$-invariant subset if and only if there is no absorbing subset $B$ of $G$ such that $A \approx B$. 
\end{proposition}

\noindent If $A$ is equivalent to $G$, then $X_A = \beta G$ is compact and, of course, \emph{all} closed subsets of $X_A$ are compact. Any infinite group contains many absorbing subsets that are not equivalent to the group itself (see Example~\ref{ex:four}, Case III). For a quick example, take $G = \Z$ and $A = \N$. 

This section is devoted to answer the second question in the case of discrete minimal subsets (Theorem~\ref{thm:discrete-vs-infdiv}). At the same time we will obtain a necessary and sufficient condition on the set $A$ ensuring that all minimal closed $G$-invariant subsets of $X_A$ are neither compact nor discrete.
 
If $A$ is finite, then $X_A = G$, which is a minimal discrete $G$-space. More generally, it was shown in  \cite[Example 5.6]{KelMonRor:supramenable} that \emph{every}  minimal closed $G$-invariant subset of $X_A$ is  discrete if $|gA \cap A| < \infty$ for all $g \in G \setminus \{e\}$. It was also shown in  \cite{KelMonRor:supramenable} that each infinite group $G$ possesses an \emph{infinite} subset $A$ with this property. A complete answer to the question of when $X_A$ contains a discrete minimal closed $G$-invariant subset is  phrased in terms of a divisibility property of the set $A$: 

\begin{definition} \label{def:divisible}
Let $G$ be a group and let $A$ be a non-empty subset of $G$. For each $n \ge 1$ we say that $A$ is \emph{$n$-divisible} if there are pairwise disjoint subsets $\{A_j\}_{j=1}^n$ of $A$ such that $A_j \approx A$  for all $j$. 

We say that $A$ is  \emph{infinitely divisible} if $A$ is $n$-divisible for all integers $n \ge 1$. 
\end{definition}

\noindent  If $A \subseteq G$ is non-empty and finite, then $A$ is $n$-divisible if and only if $|A| \ge n$, cf.\ Example~\ref{ex:subgroup1}~(ii). In Example~\ref{ex:subgroup3} below we show that each infinite group is infinitely divisible.

Recall from \eqref{eq:O(V,x)}  the definition of the subset $O_X(V,x)$  of $G$ associated with a $G$-space $X$, a subset $V$ of $X$, and an element $x$ in $X$. In the case where $X = \beta G$ and $V = K_A$ we denote this set by $O(A,x)$. In other words,
$$O(A,x) = \{g \in G \mid g.x \in K_A\}.$$

Observe that $O(A,e)=A$ and that $e \in O(A,x)$ whenever $x \in K_A$. Here are some further properties of finitely and infinitely divisible sets, and of the set $O(A,x)$:

\begin{lemma} \label{lm:O(A,x)}
 Let $A$ and $B$ be non-empty subsets of a group $G$ and let $x \in \beta G$. Then:
\begin{enumerate}
\item $O(A \cup B,x) = O(A,x) \cup O(B,x)$; and if $A$ and $B$ are disjoint, then so are $O(A,x)$ and $O(B,x)$.  \vspace{.1cm}
\item If $A \propto B$, then $O(A,x) \propto O(B,x)$.  \vspace{.1cm}
\item If $A \approx B$, then $O(A,x) \ne \emptyset \Leftrightarrow O(B,x) \ne \emptyset$.   \vspace{.1cm}
\item If $A \approx B$, then $|O(A,x)| = \infty \Leftrightarrow |O(B,x)| = \infty$.  \vspace{.1cm}
\item If $A$ is $n$-divisible, then $|O(A,x)| \ge n$ for all $x \in K_A$.  \vspace{.1cm}
\item If $A$ is infinitely divisible, then $|O(A,x)| = \infty$  for all $x \in K_A$.  
\end{enumerate}
\end{lemma}

\begin{proof} (i) follows from the identity $K_{A \cup B} = K_A \cup K_B$, and from the fact that $A \cap B = \emptyset$ implies that $K_A \cap K_B = \emptyset$, see 
\cite[Lemma 2.4(i)]{KelMonRor:supramenable}.  To see (ii), if $A \propto B$, then $A \subseteq FB$ for some finite subset $F$ of $G$. Now use that $F.K_B = K_{FB}$ for all $B \subseteq G$, cf.\ \cite[Lemma 2.4(v)]{KelMonRor:supramenable}, to conclude that $O(A,x) \subseteq F.O(B,x)$.

(iii) and (iv) follow from (ii). (v) follows from (i), (iii) and (iv) and from the fact that $O(A,x) \ne \emptyset$ whenever $x \in K_A$. (vi) follows from (v).
\end{proof}

\begin{lemma} \label{lm:2-div} Let $G$ be a discrete group and let $A$ be a non-empty subset of $G$. Then $A$ is $2$-divisible if and only if for all $x \in K_A$ there exists $g \ne e$ in $G$ such that $g.x \in K_A$.
\end{lemma}

\noindent In the language of Lemma \ref{lm:O(A,x)} this lemma says that $A$ is $2$-divisible if and only if $|O(A,x)| \ge 2$ for all $x \in K_A$.

\begin{proof} The "only if" part follows from Lemma \ref{lm:O(A,x)} (v). Let us prove the "if" part of the lemma. We show first that there are clopen subsets $U_1, U_2, \dots, U_n$ of $K_A$ and elements $g_1, g_2, \dots, g_n \in G$ such that 
\begin{itemize}
\item $K_A = U_1 \cup U_2 \cup \cdots \cup U_n$, \vspace{.1cm}
\item $g_j.U_j \cap U_j = \emptyset$ for all $j$, \vspace{.1cm}
\item $g_j.U_j \subseteq K_A$ for all $j$. \vspace{.1cm}
\end{itemize}
For each $x \in K_A$ there exists an element $g_x \ne e$ in $G$ such that $g_x.x \in K_A$. Since $K_A$ is totally disconnected there is a clopen subset $U_x$ of $K_A$ such that $g_x.U_x \subseteq K_A$, $g_x.U_x \cap U_x= \emptyset$, and $x \in U_x$.  We can now take the clopen subsets  $U_1, U_2, \dots, U_n$ of $K_A$ to be a subset of  the open cover $\{U_x\}_{x \in K_A}$ of $K_A$ which still covers $K_A$, and we take  $g_1, \dots, g_n$ to be their associated group elements. 

Put $V_1 = U_1$,  put
$$V_j = U_j \setminus \Big(\bigcup_{i=1}^{j-1} (g_i.V_i \cup g_j^{-1}.V_i) \Big),\qquad j=2,3, \dots, n,$$
and put 
$$K_1 = V_1 \cup V_2 \cup \cdots \cup V_n, \qquad K_2 = g_1.V_1 \cup g_2.V_2 \cup \cdots \cup g_n.V_n.$$
Then $K_1$ and $K_2$ are pairwise disjoint clopen subsets of $K_A$. Let $F$ be the finite subset of $G$ that consists of the elements $g_1, \dots, g_n$, their inverses, and the neutral element $e$. Then $K_A \subseteq F.K_1$ and $K_A \subseteq F^2.K_2$.

To prove that $K_A \subseteq F.K_1$, it suffices to show that $U_j \subseteq F.K_1$ for all $j$. Each $V_j$ is contained in $K_1$, which again is contained in $F.K_1$. Hence $g_i.V_i$ and $g_j^{-1}.V_i$ are contained in $F.K_1$ for all $i,j$. This implies that $U_j$ is contained in $F.K_1$ for all $j$.  

As $K_1 \subseteq F.K_2$ it follows that $K_A \subseteq F^2.K_2$.

Put $A_j = G \cap K_j$. Then $A_1$ and $A_2$ are disjoint subsets of $A$, and $K_{A_j} = K_j$ by \cite[Lemma 2.4(ii)]{KelMonRor:supramenable}. Using \cite[Lemma 2.4(v)]{KelMonRor:supramenable} we further conclude that $K_A \subseteq F.K_{A_1} = K_{FA_1}$, which by \cite[Lemma 2.4(iii)]{KelMonRor:supramenable} entails that $A \subseteq FA_1$, i.e., $A \propto A_1$. In a similar way we see that $A \propto A_2$. As $A_1$ and $A_2$ are subsets of $A$ we conclude that $A \approx A_1 \approx A_2$, so $A$ is $2$-divisible. 
\end{proof}

\begin{lemma} \label{lm:inf-div}
The following three conditions are equivalent for each non-empty subset $A$ of a discrete group $G$:
\begin{enumerate}
\item $A$ is infinitely divisible, \vspace{.1cm}
\item  $|O(A,x)| = \infty$ for all $x \in K_A$, \vspace{.1cm}
\item there are pairwise disjoint subsets $\{A_j\}_{j=1}^\infty$ of $A$ such that $A_j \approx A$  for all $j$. 
\end{enumerate}
Moreover, if $A$ and $B$ are non-empty subsets of $G$ with $A \approx B$, then $A$ is infinitely divisible if and only if $B$ is infinitely divisible.
\end{lemma}

\begin{proof} (i) $\Rightarrow$ (ii) is contained in Lemma \ref{lm:O(A,x)}  (vi).  (iii) $\Rightarrow$ (i) is trivial. 

 (ii) $\Rightarrow$ (iii). If (ii) holds, then it follows from Lemma~\ref{lm:2-div} that $A$ is $2$-divisible, so there exist pairwise disjoint subsets $A_1$ and $A_2$ of $A$ such that $A \approx A_1 \approx A_2$. Lemma~\ref{lm:O(A,x)} (iv) tells us that $|O(A_1,x)| = |O(A_2,x)| = \infty$ for all $x \in K_A$. Hence $A_1$ and $A_2$ are $2$-divisible by Lemma~\ref{lm:2-div}. We can therefore keep dividing, and we arrive at (iii).

The last statement of the lemma follows from Lemma \ref{lm:O(A,x)} (iv) and the equivalence of (i) and (ii).
\end{proof}

\begin{example} \label{ex:subgroup3} Any infinite group $G$ is infinitely divisible (as a subset of itself). This follows immediately from Lemma \ref{lm:inf-div} above, since $O(G,x) = G$ for all $x \in K_G = \beta G$. 

More generally, every infinite subgroup $H$ of $G$ is infinitely divisible. Indeed, we just saw that $H$ is infinitely divisible viewed as a subset of itself, and hence it is infinitely divisible relatively to $G$. 
\end{example}

\noindent We shall use the following standard result about crossed product \Cs s.

\begin{lemma} \label{lm:char}
Let $G$ be a discrete group acting on a locally compact Hausdorff space $X$. For each $x \in X$, let $\varphi_x$ be the state on $C_0(X) \rtimes_\red G$ given by $\varphi_x = \rho_x \circ E$, where $E \colon C_0(X) \rtimes_\red G \to C_0(X)$ is the canonical conditional expectation, and where $\rho_x \colon C_0(X) \to \C$ is point evaluation at $x$. 

Let $K$ be a compact-open subset of $X$, let $x \in K$, and suppose that $g.x \notin K$ for all $g \ne e$. The restriction of $\varphi_x$ to $1_K \big( C_0(X) \rtimes_\red G\big)1_K$ is then a  character.\footnote{A character on a unital \Cs{} $A$ is a non-zero homomorphism from $A$ into $\C$.}
\end{lemma}

\begin{proof} Observe  that $\varphi_x(1_K) = 1_K(x) = 1$, so $\varphi_x \ne 0$. To show that $\varphi_x$ is a character, it suffices to show that $\varphi_x(a^2) = \varphi_x(a)^2$ for all $a$ in the corner algebra $1_K \big( C_0(X) \rtimes_\red G\big)1_K$. Each such $a$ is a formal sum $a = \sum_{g \in G} f_g u_g$, where $f_g \in C_0(X)$ and $g \mapsto u_g$ is the unitary representation (in the multiplier algebra of $C_0(X) \rtimes_\red G$) of the action of $G$ on $C_0(X)$. Since $E(a) = f_e$, we see that $\varphi_x(a) = f_e(x)$. 

The condition that $a = 1_Ka1_K$ implies that  the support of each $f_g$ is contained in $K$. In particular, $f_h(g.x) = 0$ for all $h \in G$ and for all $g \ne e$.  Now, $E(a^2) = \sum_{g \in G} f_g \alpha_g(f_{g^{-1}})$.  Hence
$$\varphi_x(a^2) = \sum_{g \in G} f_g(x) f_{g^{-1}}(g.x) = f_e(x)^2 = \varphi_x(a)^2,$$
as desired.
\end{proof}

\begin{proposition} \label{prop:characters}
Let $G$ be a discrete group and let $A$ be a non-empty subset of $G$. Then the following are equivalent:
\begin{enumerate}
\item $1_A \big(\ell^\infty(G) \rtimes_\red G\big)1_A$ has a character.  \vspace{.1cm}
\item $A$ is not $2$-divisible. 
\end{enumerate}
\end{proposition}

\begin{proof}
(i) $\Rightarrow$ (ii). Suppose that (ii) does not hold, i.e., $A$ is the disjoint union of $A_1$ and $A_2$ where $A \approx A_1 \approx A_2$. Then $1_{A_1}$ and $1_{A_2}$ are full and pairwise orthogonal projections in $1_A \big(\ell^\infty(G) \rtimes_\red G\big)1_A$, cf.\ Lemma \ref{lm:1_A}. However, no unital $C^*$-algebra containing two full pairwise orthogonal projections can admit a character, so (i) does not hold.

(ii) $\Rightarrow$ (i). 
It follows from Lemma \ref{lm:2-div} that $K_A$ contains an element $x$ such that $g.x \notin K_A$ for all $g \ne e$. Lemma \ref{lm:char} then says that $\varphi_x$ is a character on 
$$1_{K_A} \big( C(\beta G) \rtimes_\red G\big) 1_{K_A} \cong 1_A \big(\ell^\infty(G) \rtimes_\red G\big)1_A.$$
\end{proof}

\begin{theorem} \label{thm:discrete-vs-infdiv}
Let $G$ be a discrete group and let $A$ be a non-empty subset of $G$. Then the following are equivalent:
\begin{enumerate}
\item Some minimal closed $G$-invariant subset of $X_A$ is discrete. \vspace{.1cm}
\item $1_A \big(\ell^\infty(G) \rtimes_\red G\big)1_A$ has a non-zero finite-dimensional quotient.  \vspace{0.1cm}
\item $A$ is not infinitely divisible.
\end{enumerate}
\end{theorem}

\begin{proof} Set 
$$\cA = 1_A \big(\ell^\infty(G) \rtimes_\red G\big)1_A \cong 1_{K_A} \big( C(\beta G) \rtimes_\red G\big) 1_{K_A} = 1_{K_A} \big( C_0(X_A) \rtimes_\red G\big) 1_{K_A},$$ 
and observe that $1_{K_A}$ is a full projection in $C_0(X_A) \rtimes_\red G$. 

(i) $\Rightarrow$ (ii). Suppose that $Z$ is a discrete minimal closed invariant subset of $X_A$.  Then $Z$ is a discrete minimal and free $G$-space, whence $Z$ is isomorphic to $G$ (viewed as a $G$-space with respect to left multiplication). The inclusion mapping $Z \hookrightarrow X_A$, which is a $G$-map, induces a $G$-equivariant surjection $C_0(X_A) \to C_0(Z)$, which again induces  a  surjective $^*$-homomorphism 
$$C_0(X_A) \rtimes_\red G \to C_0(Z) \rtimes_\red G \, \cong \, C_0(G) \rtimes_\red G \, \cong \, \cK,$$
where $\cK$ is the algebra of compact operators (on $\ell^2(G)$).  Hence there is a surjective $^*$-homomorphism from $\cA$ onto a unital corner of $\cK$, which is a finite dimensional matrix algebra.

(ii) $\Rightarrow$ (i). If (ii) holds, then  there is a proper closed two-sided ideal $I$ in $C_0(X_A) \rtimes_\red G$ such that the quotient is stably isomorphic to a finite-dimensional \Cs. Upon replacing $I$ with a larger ideal, we may assume that $I$ is a maximal proper ideal and that the quotient is isomorphic to  the algebra of compact operators $\cK(H)$ on a Hilbert space $H$.

The intersection $C_0(X_A) \cap I$ is an invariant ideal in $C_0(X_A)$ and therefore equal to $C_0(U)$ for some open invariant subset $U$ of $X_A$. Since $I$ is a maximal ideal in $C_0(X_A) \rtimes_\red G$ it follows that $U$ is a maximal proper open invariant subset of $X_A$, so $Z = X_A \setminus U$ is a minimal closed invariant subset of $X_A$. The kernel of the map $C_0(X_A) \to \big(C_0(X_A) \rtimes_\red G\big)/I$ is equal to $C_0(U)$, so we have an injective $^*$-homomorphism
$$C_0(Z) \to \big(C_0(X_A) \rtimes_\red G\big)/I \; \cong \; \cK(H).$$
Hence $Z$ must be discrete.

(ii) $\Rightarrow$ (iii). Suppose that $A$ is infinitely divisible. Find a sequence $\{A_j\}_{j=1}^\infty$ of pairwise disjoint subsets of $A$ with $A \approx A_j$ for all $j$. It follows from Lemma~\ref{lm:1_A} that the projections $1_{A_j}$ are full in $\cA$, and they are also pairwise orthogonal. However, no unital \Cs{} which contains a sequence of pairwise orthogonal full projections can have a (non-zero) finite dimensional quotient. Hence (ii) does not hold.

(iii) $\Rightarrow$ (ii). If $A$ is not infinitely divisible, then there exists a subset $B$ of $A$ such that $B$ is not $2$-divisible and $A \approx B$. (Indeed,  let $n \ge 1$ be the maximal number for which there exist pairwise disjoint subsets $\{A_j\}_{j=1}^n$  of $A$ with $A_j \approx A$ for all $j$. Put $B=A_1$.)

Let $p = 1_B \in \cA$, and use Lemma \ref{lm:1_A} to see that $p$ is a full projection in $\cA$. It follows from Proposition \ref{prop:characters} that $p\cA p$ admits a character. In other words, $(p+I)(\cA/I)(p+I)$ is one-dimensional for some closed two-sided ideal $I$ in $\cA$. Thus $\cA/I$ is a unital \Cs{} which is Morita equivalent to $\C$, which entails that $\cA/I$ is finite dimensional. 
\end{proof}

\noindent It is easy to decide when subsets of \emph{minimal} type of a group are infinitely divisible, respectively, equivalent to an absorbing set:

\begin{corollary} \label{cor:min-infdiv}
Let $G$ be an infinite group and let $A$ be a subset of $G$ of minimal type. Then 
\begin{enumerate}
\item $A$ is infinitely divisible if and only if $A$ is infinite.
\item $A$ is equivalent to an absorbing set if and only if $A$ is equivalent to $G$. 
\end{enumerate}
\end{corollary}

\begin{proof} 
(i). "Only if" is clear. Suppose that $A$ is infinite. For each minimal closed invariant subset $Z$ of $X_A$, there is $z_0 \in Z$ such that $\Type(G,Z,z_0) = [A]$, 
cf.\ Lemma~\ref{lm:type-of-universal}. As $[A] \ne [\{e\}]$, it follows from Proposition \ref{prop:discrete} that $Z$ is not discrete. Hence $A$ must be infinitely divisible by Theorem~\ref{thm:discrete-vs-infdiv}.

(ii). "Only if" is clear. Suppose that $A$ is equivalent to an absorbing set. Then there is a compact minimal closed invariang subset $Z$ of $X_A$ by \cite[Proposition 5.5 (iv)]{KelMonRor:supramenable}, and there is $z_0 \in Z$ such that $\Type(G,Z,z_0) = [A]$,  cf.\ Lemma~\ref{lm:type-of-universal}. But then $A \approx G$ by Proposition~\ref{prop:compacttype}. 
\end{proof}

%
%
%

\noindent There is an abundance of subsets $A$ of $G$ which are infinite but not infinitely divisible, see  Example~\ref{ex:discrete-compact} below for a method to construct such sets. None of these sets are of minimal type by Corollary~\ref{cor:min-infdiv}.

\begin{example}[Four classes of subsets of a group] \label{ex:four}
Let $G$ be an infinite group. We describe here four classes of subsets $A$ of $G$ depending on whether or not $A$ is infinitely divisible, respectively, equivalent to an absorbing set. All four cases occur in all infinite groups. The stated properties of the minimal closed $G$-invariant subsets of $X_A$ follow from Theorem~\ref{thm:discrete-vs-infdiv} and from \cite[Proposition 5.5 (iv)]{KelMonRor:supramenable} (see also Proposition~\ref{prop:absorbing}).

{\bf{Case I:}} \emph{$A$ is infinitely divisible and absorbing.} This is satisfied for example with $A=G$.

In this case no minimal closed $G$-invariant subset of $X_A$ is discrete, while at least some minimal closed $G$-invariant subsets of $X_A$ are compact (they are all compact if $A=G$).  

{\bf{Case II:}} \emph{$A$ is infinitely divisible and not equivalent to an absorbing set.} Examples of such sets in an arbitrary infinite countable group $G$ is given in Theorem~\ref{thm:A} in the next section. 

In this case every minimal closed $G$-invariant subset of $X_A$ is non-compact and non-discrete. The existence of such subsets $A$ therefore gives the existence of a free minimal action of $G$ on a non-compact and non-discrete locally compact Hausdorff space. We can take this space to be the non-compact locally compact Cantor set $\Can^*$, cf.\ Theorem~\ref{thm:min-action-2}.   

{\bf{Case III:}} \emph{$A$ is absorbing but not infinitely divisible.} Examples of such sets in an arbitrary infinite countable group $G$ are given in Example~\ref{ex:discrete-compact} below.  

In this case some minimal closed $G$-invariant subset of $X_A$ is compact while some other minimal closed $G$-invariant subset of $X_A$ is discrete. The existence of such sets $A$ shows that Corollary~\ref{cor:isomorphism} does not hold in general (when $A$ is not of minimal type) and that the first question posed in the beginning of this section has a negative answer in general. 

{\bf{Case IV:}} \emph{$A$ is not infinitely divisible and not equivalent to an absorbing set.} This is satisfied for example with $A = \{e\}$, and also for any (finite or infinite) subset $A$ of $G$ such that $|A \cap gA| < \infty$ for all $g \in G \setminus \{e\}$. Each infinite group contains infinite subsets $A$ with this property, see \cite[Lemma 3.8]{KelMonRor:supramenable} and its proof.

In this case no minimal closed $G$-invariant subset of $X_A$ is compact, while at least some minimal closed $G$-invariant subset of $X_A$ is discrete. If $A$ satisfies $|A \cap gA| < \infty$ for all $g \in G \setminus \{e\}$, then \emph{all} minimal closed $G$-invariant subsets of $X_A$ are discrete by \cite[Example 5.6]{KelMonRor:supramenable}. 
\end{example}

\begin{remark} \label{rem:construct-inf-div} 
We present here a general method to construct an infinitely divisible subset in an arbitrary group $G$. 

Choose first an infinite sequence $\{a_n\}_{n=1}^\infty$ of elements in $G$ 
so that the set
\[
F_n=\{b_1b_2\dots b_n \mid b_i=e \text{ or }b_i=a_i\}
\]
has cardinality $2^n$ for each $n$.  Observe that $F_1 = \{e,a_1\}$ and that $F_{n+1} = F_n \cup F_n\, a_{n+1}$ for $n \ge 1$. It follows that $|F_n|=2^n$ for all $n$ if the $a_n$'s are chosen such that $F_n \cap F_n\, a_{n+1} = \emptyset$ for all $n$, i.e., such that 
$$a_{n+1} \in G \setminus F_n^{-1}F_n.$$

Take now another infinite sequence $\{c_n\}_{n=1}^\infty$ of elements in $G$ 
so that the sets $F_nc_n$, $n \in \N$, are mutually disjoint. Then 
\begin{equation} \label{eq:A-infdiv}
A=\bigcup_{n=1}^\infty F_nc_n
\end{equation}
is infinitely divisible. 
Let us check this. 
Fix an integer $n\ge 1$. 
For $m > n$, define 
\[
G_m=\{b_{n+1}b_{n+2}\dots b_{m} \mid b_i=e\text{ or }b_i=a_{i}\}. 
\]
Note that $F_nG_m = F_{m}$. Set 
\[
B=\bigcup_{m= n+1}^\infty G_mc_m. 
\]
Then $\{aB\}_{a \in F_n}$ are  pairwise disjoint subsets of $A$. Since $A\setminus F_nB$ is a finite set, we see that $B \approx A$, and hence that $aB \approx A$ for all $a \in F_n$. Therefore $A$ is $2^n$-divisible. As $n \ge 1$ was arbitrary we conclude that $A$ is infinitely divisible. 
\end{remark}

\noindent  The lemma below, whose proof easily follows from Lemma \ref{lm:approx-geom}, contains an easy criterion for not being $2$-divisible. 

\begin{lemma} \label{lm:2-div-condition} Let $G$ a countable group and let $d$ be a proper right-invariant metric on $G$. Suppose that $A$ is a subset of $G$ such that for each $C< \infty$ there exists $g \in A$ with 
$$d(g,A\setminus \{g\}) > C.$$
Then $A$ is not $2$-divisible.
\end{lemma}

\begin{example} \label{ex:discrete-compact} We give here examples of subsets $A$ of any infinite countable group $G$ which are absorbing but not infinitely divisible, cf.\ Example~\ref{ex:four} (Case III).

For a first easy example of such a set, take $G = \Z$ and 
$$A = \{-n^2 \mid n \in \N\} \cup \N.$$
Then $A$ is absorbing, cf.\  Definition~\ref{def:absorbing}. Use Lemma~\ref{lm:2-div-condition} to see that $A$ is not $2$-divisible.

Consider now an arbitrary infinite countable group $G$ and choose a proper right-invariant metric $d$ on $G$. Write $G = \bigcup_{n=1}^\infty F_n$, where $\{F_n\}_{n=1}^\infty$ is an increasing sequence of finite subsets of $G$. We claim that there exist sequences $\{g_n\}_{n=1}^\infty$ and $\{h_n\}_{n=1}^\infty$ in $G$ such that if
$$A = \{g_1,g_2,g_3, \dots \} \cup \bigcup_{n=1}^\infty F_n h_n,$$
then $d(g_n,A \setminus \{g_n\}) \ge n$ for all $n$. It will then follow from Lemma \ref{lm:2-div-condition} that $A$ is not $2$-divisible; and we see directly from Definition~\ref{def:absorbing} that $A$ is absorbing.

Note first that if $F,F'$ are finite subsets of the infinite group $G$ and if $C< \infty$, then there exists $g \in G$ such that $d(Fg,F') \ge C$. We can now construct $\{g_n\}$ and $\{h_n\}$ as follows. Put $g_1=e$ and
choose $h_1$ such that $d(F_1h_1,g_1) \ge 1$. Next, choose $g_2$ and then $h_2$ such that 
$$d\big(g_2, \{g_1\} \cup F_1h_1\big) \ge 2, \qquad 
d\big(F_2h_2, \{g_1,g_2\}\big) \ge 2.$$
Continue in this way. At stage $n$ we find $g_n, h_n \in G$ such that
$$d\big(g_n, \{g_1,g_2, \dots, g_{n-1}\} \cup F_1h_1\cup \cdots \cup F_{n-1}h_{n-1}\big) \ge n,$$ 
$$d\big(F_nh_n, \{g_1,g_2, \dots, g_{n}\} \big) \ge n.$$
Then $A$ has the desired properties. 
\end{example}


\section{Minimal $G$-spaces that are non-discrete and non-compact} 
\label{sec:non-discrete-compact} 

\noindent We show here that each countable infinite group contains an infinite divisible subset  which is not equivalent to an absorbing set. This, in turn, gives the existence of a minimal free action of the  group on a non-compact and non-discrete locally compact Hausdorff space. This space can, moreover, be taken to be the non-compact locally compact Cantor set.

Before constructing the general examples, let us look at a concrete example:

\begin{example} \label{ex:subgroup4} Let $G$ be a group and let $H$ be a subgroup of $G$. We know from Example \ref{ex:subgroup-2} that $H$ is of minimal type.  It follows from 
Lemma~\ref{lm:type-of-universal} and Corollary~\ref{cor:isomorphism} that each minimal closed invariant subset of $X_H$ is of type $[H]$ (with respect to a suitable base point), and that all such minimal closed invariant sets are isomorphic as $G$-spaces. 

We know from Example \ref{ex:subgroup3} that each infinite subgroup is infinitely divisible. We show below that no subgroup of infinite index is equivalent to an absorbing set. It follows that if $H$ is an infinite subgroup of $G$ of infinite index, then each minimal closed $G$-invariant subset of $X_H$ is non-compact and non-discrete. Each of these invariant subsets of $X_H$ is then a free minimal $G$-spaces, which is neither compact nor discrete. Not all infinite groups have an infinite subgroup of infinite index. The Tarski monsters, whose existence was proved by Olshanskii, have no proper infinite subgroups.

Let us show  that no subgroup $H$ of $G$ of infinite index is  equivalent to an absorbing set. We must show that $FH$ is not absorbing for any finite subset $F$ of $G$, i.e., that $\bigcap_{s \in S} sFH = \emptyset$ for some finite subset $S$ of $G$.  Write $F = \{f_1,f_2, \dots, f_n\}$. Since $FH$ is contained in the union of finitely many left-cosets of $H$ there exists $x \in G$ such that $FH \cap xH = \emptyset$. Put 
$$S = \{e, f_1x^{-1}, f_2x^{-1}, \dots, f_nx^{-1}\}.$$
Note that $f_j x^{-1}FH \cap f_jH = \emptyset$. Hence
$$\bigcap_{s \in S} sFH \; \subseteq \; FH \cap \bigcap_{j=1}^n f_jx^{-1}FH \; \subseteq \; \bigcup_{j=1}^n \big(f_jH \cap f_jx^{-1}FH\big) \; = \; \emptyset,$$
as desired.
\end{example}

\begin{theorem} \label{thm:A}
In any countably infinite group $G$ there exists an infinitely divisible subset $A$ which is not equivalent to an absorbing subset of $G$. 
\end{theorem}

\begin{proof} Let  $\{T_n\}_{n=1}^\infty$ be an increasing sequence of finite subsets of $G$ 
whose union is $G$. Following the recipe of Remark \ref{rem:construct-inf-div} we will construct sequences $\{a_n\}_{n=1}^\infty$,  $\{c_n\}_{n=1}^\infty$  and $\{d_n\}_{n=1}^\infty$ of elements in $G$, and sequences $\{A_n\}_{n=1}^\infty$ and $\{F_n\}_{n=1}^\infty$ of finite subsets of $G$ as follows. For $n=1$ choose $a_1,c_1, d_1, F_1$ and $A_1$ such that 
\begin{itemize}
\item[(a)] $a_1 \ne e$, $F_1 =\{e,a_1\}$, $c_1 = e$, $A_1 = F_1$, $d_1 \notin T_1A_1A_1^{-1}T_1^{-1}$. \vspace{.1cm}
\end{itemize}
This is clearly possible. For all $n \ge 1$, the elements $a_{n+1},b_{n+1},c_{n+1}$ and the sets $A_{n+1}$ and $F_{n+1}$ are constructed inductively so that they satisfy:
\begin{itemize}
\item[(b)] $F_{n+1} = F_{n} \cup F_{n} \, a_{n+1}$ (disjoint union). \vspace{.1cm}
\item[(c)] $A_{n+1} = A_{n} \cup F_{n+1} \, c_{n+1} = \bigcup_{k=1}^{n+1} F_k \, c_k$ (disjoint unions). \vspace{.2cm}
\item[(d)] $d_k  \notin T_kA_{n+1}A_{n+1}^{-1}T_k^{-1}$ for $1 \le k \le n+1$. \vspace{.1cm}
\end{itemize}
Given $n \ge 1$, and let us show how to find $a_{n+1},c_{n+1},d_{n+1}, F_{n+1}$ and $A_{n+1}$ (when the previous elements and sets have been found). 

Choose  $a_{n+1} \in G$ outside the finite set
$$ F_{n}^{-1}F_{n} \; \cup \; \bigcup_{k=1}^{n} F_{n}^{-1}T_k^{-1}\{d_k, d_k^{-1}\}T_kF_n,$$
and set $F_{n+1} = F_{n} \cup F_{n} \, a_{n+1}$. Then (b) holds. Moreover
\begin{equation} \label{eq:g1}
d_k\notin T_kF_n\{a_{n+1},a_{n+1}^{-1}\}F_n^{-1}T_k^{-1}, \quad 1\le k \le n.
\end{equation}
Since $F_nc_n$ is a subset of $A_n$, it follows from (d) that 
\begin{equation} \label{eq:g2}
d_k\notin T_kF_nF_n^{-1}T_k^{-1}, \quad  1 \le k \le n.
\end{equation}
By Equations \eqref{eq:g1} and \eqref{eq:g2} and the definition of $F_{n+1}$ it follows that 
\begin{equation} \label{eq:g3}
d_k\notin T_kF_{n+1}F_{n+1}^{-1}T_k^{-1}, \quad 1 \le k \le n.
\end{equation}

Choose $c_{n+1}\in G$ outside the finite set
$$F_{n+1}^{-1}A_n \, \cup \, \bigcup_{k=1}^n \big(A_n^{-1}T_k^{-1}d_kT_kF_{n+1} \, \cup \, F_{n+1}^{-1}T_k^{-1} d_kT_kA_n\big),$$
and set $A_{n+1}=A_n\cup F_{n+1}c_{n+1}$. Then (c) holds, and
\begin{equation} \label{eq:g4}
d_k \notin T_kF_{n+1}c_{n+1}A_n^{-1}T_k^{-1} \, \cup \, T_kA_nc_{n+1}^{-1}F_{n+1}^{-1}T_k^{-1}, \quad 1 \le k \le n.
\end{equation}
It follows from Equations \eqref{eq:g2}, \eqref{eq:g3} and \eqref{eq:g4} that (d) is satisfied for $1 \le k \le n$. Choose $d_{n+1}\in G$ outside the finite set $T_{n+1}A_{n+1}A_{n+1}^{-1}T_{n+1}^{-1}$. Then (d) holds also for $k=n+1$. This completes the construction of the sequences $\{a_n\}_{n=1}^\infty$,  $\{c_n\}_{n=1}^\infty$, $\{d_n\}_{n=1}^\infty$, $\{A_n\}_{n=1}^\infty$ and $\{F_n\}_{n=1}^\infty$. 

Set
$$A = \bigcup_{n=1}^\infty A_n = \bigcup_{n=1}^\infty F_n\, c_n.$$
Then $A$ is infinitely divisible by Remark \ref{rem:construct-inf-div}. 
Condition (d) (and the fact that the $A_n$'s are increasing) imply that 
$d_k$ does not belong to $T_kAA^{-1}T_k^{-1}$. 
Hence,  $T_kA\cap d_kT_kA=\emptyset$. 
This means that $T_kA$ is not absorbing. As every finite subset $F$ of $G$ is contained in $T_k$ for some $k$ we can conclude that 
$A$ is not equivalent to an absorbing set. 
\end{proof}

\begin{corollary} \label{cor:min-action-1}
Every countable infinite group $G$ admits a free minimal action on a totally disconnected locally compact Hausdorff space, which is neither discrete nor compact. The action can, moreover, be assumed to be amenable if $G$ is exact.
\end{corollary}

\begin{proof} Let $G$ be a given countable infinite group, and let $A \subseteq G$ be as in Theorem \ref{thm:A}. Let $Z$ be any minimal closed invariant subset of $X_A$. Then $Z$ is non-compact by  Proposition~\ref{prop:absorbing} (cf.\ \cite[Proposition 5.5(iv)]{KelMonRor:supramenable}), since $A$ is not equivalent to an absorbing set; and $Z$ is non-discrete by Theorem \ref{thm:discrete-vs-infdiv}. 

It follows from \cite[Proposition 5.5]{KelMonRor:supramenable} that $Z$ is a locally compact totally disconnected Hausdorff space, and that $G$ acts freely on $Z$. It further follows from \cite[Proposition 5.5]{KelMonRor:supramenable} that the action of $G$ on $Z$ is amenable if $G$ is exact (essentially by Ozawa's result, that $G$ acts amenably on $\beta G$ when $G$ is exact). 
\end{proof}

\noindent Recall the definition of minimal types from Definition \ref{def:minimal-type}.

\begin{corollary} \label{cor:min-type} Each countable infinite group $G$ contains minimal types different from $[G]$ (the compact type) and $[\{e\}]$ (the discrete type). 
\end{corollary}

\begin{proof} Let $Z$ be a locally compact Hausdorff space, which is neither discrete nor compact, on which $G$ acts minimally (and freely). Pick any $z_0 \in Z$ and set $[A] = \Type(G,Z,z_0)$. Then $[A]$ is of minimal type, and $[A] \ne [G]$ (because $Z$ is non-compact) and $[A] \ne [\{e\}]$ (because $Z$ is non-discrete).
\end{proof}

\noindent We wish to replace $Z$ above with a second countable $G$-space, and need for this the following lemma:

\begin{lemma} \label{lm:subalgebra} Let $G$ be a countable group, and let $Z$ be a minimal locally compact  free $G$-space. Let $\cA$ be a non-zero $G$-invariant closed sub-\Cs{} of $C_0(Z)$, and let $Y= \widehat{\cA}$ denote the spectrum of $\cA$ (so that $\cA = C_0(Y)$). 

Then $Y$ is a locally compact $G$-space, and the inclusion $\cA \subseteq C_0(Z)$ arises from a surjective proper continuous $G$-map $\varphi \colon Z \to Y$. Moreover:
\begin{enumerate}
\item $Y$ is non-compact if $Z$ is non-compact. \vspace{.1cm}
\item The action of $G$ on $Y$ is minimal. \vspace{.1cm}
\item $Y$ is non-discrete if the action of $G$ on $Y$ is free and $Z$ is non-discrete.
\end{enumerate}
 \end{lemma}

\begin{proof} 
The spectrum $Y$ of the commutative \Cs{}  $\cA$ is a locally compact Hausdorff space; and the action of $G$ on $\cA$ comes from an action of $G$ on $Y$. Since $\cA$ is non-zero and $G$-invariant, and $C_0(Z)$ by assumption is $G$-simple, it follows that $\cA$ is not contained in a proper ideal of $C_0(Z)$ (i.e., for each $z \in Z$ there exists $f \in \cA$ such that $f(z) \ne 0$). The inclusion $\cA \to C_0(Z)$ is therefore induced by a surjective proper continuous  $G$-map $\varphi \colon Z \to Y$. 

(i). If $Y$ is compact, then $Z=\varphi^{-1}(Y)$ is compact because $\varphi$ is proper.

(ii). If $Y_0$ is a closed $G$-invariant subspace of $Y$, then $\varphi^{-1}(Y_0)$ is a closed $G$-invariant subspace of $Z$. Hence $\varphi^{-1}(Y_0)$ is either empty or equal to $Z$, so $Y_0$ is either empty or equal to $Y$.

(iii). Suppose that $Y$ is discrete and that $G$ acts freely on $Y$. Let $z \in Z$. We claim that $G.z$ is closed in $Z$. Indeed, suppose that $z_0$ belongs to the closure of $G.z$. Then $g_\alpha.z \to z_0$ for some net $\{g_\alpha\}$ in $G$, whence $g_\alpha.\varphi(z) \to \varphi(z_0)$ in $Y$. As $Y$ is discrete, this entails that $g_\alpha.\varphi(z)$ is eventually constant; and as $G$ acts freely on $Y$ we conclude that $\{g_\alpha\}$ is eventually constant. This, of course, implies that $z_0 \in G.z$. 

Since $Z$ is minimal, $G.z = Z$ for (some/any) $z \in Z$. Hence $Z$ must have an isolated point by Baire's theorem, being a countable locally compact Hausdorff space, and therefore each point of $Z$ is isolated.
\end{proof}

\noindent The Cantor set $\Can$ is the unique compact Hausdorff space which is second countable, totally disconnected and without isolated points. There is also a unique \emph{locally compact, non-compact} Hausdorff space which is second countable, totally disconnected and without isolated points, denoted $\Can^*$ and referred to as the non-compact locally compact Cantor set. It arises, for example, by removing one point from $\Can$, or as $\Can \times \N$, or as the $p$-adic numbers.  

It was shown in \cite{KelMonRor:supramenable} that "most" countable infinite groups admit a free minimal action on $\Can^*$. It remained in \cite{KelMonRor:supramenable}  an open question if \emph{all} countable infinite groups admit such an action. We can now answer this question in the affirmative:

\begin{theorem} \label{thm:min-action-2}
Every countable infinite group $G$ admits a free minimal action on the non-compact locally compact Cantor set $\Can^*$. If $G$ is exact, then the action can, moreover, be taken to be amenable. 

For each $A \subseteq G$ of minimal type, such that $[A] \ne [G]$ and $[A] \ne [\{e\}]$, there is a free minimal action of $G$ on $\Can^*$ and $x_0 \in \Can^*$ such that $\Type(G,\Can^*,x_0) = [A]$. Again, if $G$ is exact, then the action can be taken to be amenable. 
\end{theorem}

\begin{proof} Let $A$ be a subset of $G$ of minimal type such that $[A] \ne [G]$ and $[A] \ne [\{e\}]$, whose existence is ensured by Corollary \ref{cor:min-type}, and let $Z$ be a minimal closed invariant subset of $X_A$. Then $\Type(G,Z,z_0) = [A]$ for some $z_0 \in Z$ by Lemma~\ref{lm:type-of-universal}. Observe that $Z$ is non-compact and non-discrete by the assumption that $[A] \ne [G]$ and $[A] \ne [\{e\}]$, cf.\ Proposition \ref{prop:compacttype} and Proposition \ref{prop:discrete}. 

This $G$-space, $Z$, satisfies all the axioms of the locally compact non-compact Cantor set, except it is not second countable. We shall replace $Z$ by a $G$-equivariant quotient of $Z$ which is second countable, and which retains all the other properties of the $G$-space $Z$. 

We construct the quotient space of $Z$ at the level of algebras. At this level  we seek a separable $G$-invariant subalgebra $\cA$ of $C_0(Z)$. The spectrum, $\widehat{\cA}$, of $\cA$ is then a second countable $G$-space. We arrange, moreover, that 
$\cA$ is generated by a set of projections (this will imply that $\widehat{\cA}$ is totally disconnected); and that the action of $G$ on $\widehat{\cA}$ is free, and also amenable if $G$ is exact.  It will then follow from Lemma \ref{lm:subalgebra} that $\widehat{A}$ is non-compact, non-discrete and minimal as a $G$-space. 

We can therefore, once $\cA$ has been constructed, identify $\widehat{\cA}$ with $\Can^*$, and we get a $G$-action on $\Can^*$ (arising from the action of $G$ on $\cA$) with the desired properties. 
The quotient mapping $\varphi \colon Z \to \Can^*$, which induces the inclusion $\cA \subseteq C_0(Z)$, cf.\ Lemma \ref{lm:subalgebra}, will be a proper continuous surjective $G$-map. It follows from Proposition \ref{prop:sametype} that $\Type(G,\Can^*,x_0) = \Type(G,Z,z_0) = [A]$ when  $x_0 = \varphi(z_0)$. 

We construct the separable subalgebra $\cA$ of $C_0(Z)$ following the ideas of \cite{RorSie:pi} and \cite{KelMonRor:supramenable}. The set-up here is a little different because we seek a quotient of the space $Z$, whereas in the mentioned references quotients of $X_A$ were constructed. 

Observe that $C_0(Z)$ has a countable approximate unit consisting of projections, call it $\{p_n\}_{n=1}^\infty$. Indeed, 
since $Z$ is totally disconnected and Hausdorff it contains a compact-open subsets $K$. By minimality we must have $Z = \bigcup_{g \in G} g.K = \bigcup_{n=1}^\infty K_n$, when $K_n = \bigcup_{g \in F_n} g.K$ for some increasing sequence $\{F_n\}$ of finite subsets of $G$ with $\bigcup F_n = G$. We can thus take  $p_n = 1_{K_n} \in C_0(Z)$.

We use  \cite[Lemma 6.1]{KelMonRor:supramenable}  to construct $\cA$ so the action of $G$ on $\widehat{\cA}$ becomes free. 
Since $G$ acts freely on $Z$, we can apply the proof of \cite[Lemma 6.1]{KelMonRor:supramenable} (with respect to the approximate unit $\{p_n\}$) to obtain a countable subset $M'$ of $C_0(Z)$ such that whenever $\cA$ is a $G$-invariant sub-\Cs{} of $C_0(Z)$ which contains $\{p_n\}$ and $M'$, then $G$ acts freely on $\widehat{\cA}$. 

Assume  that $G$ is exact. As in the proof of \cite[Lemma 6.2]{KelMonRor:supramenable}, and using the fact that the action of $G$ on $Z$ is amenable, there is a countable subset $M''$ of $C_0(Z)$ such that the action of $G$ on $\widehat{\cA}$ is amenable whenever $\cA$ contains $M''$. 

It follows from \cite[Proposition 6.6]{KelMonRor:supramenable} or \cite[Lemma 6.7]{RorSie:pi} that there exists a countable $G$-invariant family, $P$, of projections in $C_0(Z)$, such that $\cA := C^*(P)$ contains  $\{p_n\} \cup M' \cup M''$ if $G$ is exact, and $\{p_n\} \cup M'$ otherwise. This completes the proof.
\end{proof}


\section{Removing the base point from the invariant} \label{sec:basepoint}

\noindent In this last section we investigate universal properties of locally compact $G$-spaces without making reference to a base point. This leads to a new order and equivalence relations on the power set of $G$, and we define the (base point free) type of a locally compact $G$-space in terms of this new equivalence relation. It turns out that a subset is of minimal type (it represents a minimal $G$-space) if and only if it is minimal with respect to this new ordering on $P(G)$.  As a byproduct of these efforts we give in Corollary \ref{cor:char-same-type} an answer to the following question: Suppose that $X_1$ and $X_2$ are minimal locally compact $G$-spaces. When does there exists a minimal locally compact $G$-space $X$ with surjective proper continuous $G$-maps $X \to X_j$ for $j=1,2$? 

To motivate the discussion below, consider a (countable) group $G$ and a locally compact $G$-space $X$ with dense orbits. Let $K \subseteq X$ be a $G$-regular compact set, and let $x \in X$ be such that $G.x$ is dense in $X$. Then 
$$\Type(G,X,x) = [A], \quad \text{where} \; A = O_X(K,x),$$
by Proposition \ref{prop:type-condition}. If we replace  the base point $x$ with $g.x$, then $O_X(K,g.x) = O_X(K,x)g$, so $\Type(G,X,g.x) = [Ag]$.

If $y \in X$ is an arbitrary point, then $g_n.x \to y$ for some sequence (or net) $\{g_n\}_{n=1}^\infty$ in $G$. We show below that $O_X(K,y)$ is (equivalent) to the limit of the sets $\{Ag_n\}$. 

Let $G$ be an infinite countable group. 
A sequence (or a net) $\{A_n\}$ of subsets of $G$ converges to a subset $B\subseteq G$ 
if the functions $1_{A_n}$ tend to $1_B$ pointwise. 
In other words, $A_n\to B$ if and only if the identity $A_n\cap F=B\cap F$ holds eventually 
for any finite set $F$.
We may identify $1_A$ with a point in $\{0,1\}^G$, 
which is a Cantor set when equipped with the product topology. 
The pointwise convergence is exactly the same as 
the convergence with respect to this topology. 
In particular, any sequence (or net) $\{A_n\}_{n=1}^\infty$ has a convergent subsequence (or subnet), 
because $\{0,1\}^G$ is a (metrizable) compact space.

\begin{lemma} \label{lm:equiv-i}
Let $G$ be a countable infinite group, let $A$ and $B$ be subsets of $G$, and let $\{g_n\}_{n=1}^\infty$ be a sequence in $G$. 
Suppose that $\{Ag_n\}_{n=1}^\infty$ converges to $B$. 
\begin{enumerate}
\item For any finite subset $F\subseteq G$, 
the sequence $\{FAg_n\}_{n=1}^\infty$ converges to $FB$. \vspace{.1cm}
\item Let $A'$ and $B'$ be subsets of $G$. 
If $A\approx A'$ and $\{A'g_n\}_{n=1}^\infty$ con\-verges to $B'$, then $B\approx B'$. 
\end{enumerate}
\end{lemma}

\begin{proof}
(i). It is easy to see that $hAg_n$ converges to $hB$ for any $h\in G$. As
\[
1_{FA}=1_G \wedge \sum_{h\in F}1_{hA}
\]
for all finite subsets $F$ of $G$ and for all subsets $A$ of $G$ (where $\wedge$ denotes taking minimum), we see that 
\[
\lim_{n\to \infty} 1_{FAg_n} = \lim_{n\to \infty} \big(1_G \wedge \sum_{h\in F}1_{hAg_n} \big)= 1_G \wedge \sum_{h\in F}1_{hB}  =1_{FB}. 
\]

(ii).
There exists a finite set $F\subseteq G$ 
such that $A'\subseteq FA$ and $A\subseteq FA'$. 
By (i), we have $FAg_n\to FB$. 
This, together with $1_{A'}\leq 1_{FA}$, implies that $1_{B'}\leq 1_{FB}$. 
Thus $B'\subseteq FB$. 
In the same way, we get $B\subseteq FB'$. 
Hence $B\approx B'$. 
\end{proof}

\begin{definition} \label{def:equiv-i}
Let $G$ be a countable infinite group, and let $A$ and $B$ be subsets of $G$.
Write $A\succsim B$ if there exist a sequence $\{g_n\}_{n=1}^\infty$ in $G$ and a subset $B'$ of $G$ such that $Ag_n\to B'$ and $B\approx B'$. 

When $A\succsim B$ and $B\succsim A$, we write $A\sim B$. 
\end{definition}

\begin{lemma} \label{def:equiv-i-def}
Let $A,B,C$ be subsets of $G$. 
If $A\succsim B$ and $B\succsim C$, then $A\succsim C$. 
In particular, the relation $\sim$ is an equivalence relation. 
\end{lemma}

\begin{proof}
Suppose that $A\succsim B$ and $B\succsim C$. Then
there exist a sequence $\{g_n\}_{n=1}^\infty$ and a subset $B'$ 
such that $Ag_n\to B'$ and $B\approx B'$; and
there exist a sequence $\{h_n\}_{n=1}^\infty$ and a subset $C'$ 
such that $Bh_n\to C'$ and $C\approx C'$. 

The sequence $\{B'h_n\}_{n=1}^\infty$ has a convergent subsequence. 
By taking a subsequence of $\{h_n\}_{n=1}^\infty$, 
we may assume that $B'h_n \to C''$. 
By Lemma~\ref{lm:equiv-i} we have $C'\approx C''$. 

Write $G = \bigcup_{k=1}^\infty F_k$, where $\{F_k\}$ is an increasing sequence of finite subsets of $G$. Let $k \ge 1$. 
There exists $i\ge 1$ such that $B'h_i\cap F_k=C''\cap F_k$. 
Since $Ag_n\to B'$, we have $Ag_nh_i\to B'h_i$, so  there exists $j\ge 1$ such that $Ag_jh_i\cap F_k=B'h_i\cap F_k$. It follows that $Af_k\cap F_k=C''\cap F_k$ when $f_k = g_jh_i$. Hence $Af_k \to C'' \approx C$, so $A \succsim C$. 
\end{proof}

\begin{definition} For each countable infinite group $G$ let $P_\sim(G)$ be the quotient space $P(G)/\!\! \sim$. For each $A \subseteq G$ denote by $\langle A \rangle$ its equivalence class in $P_\sim(G)$.
Note that $P_\sim(G)$ becomes a partially ordered set when equipped with the order $\precsim$. 
\end{definition}

\noindent The lemma below follows easily from Definition \ref{def:equiv-i} and from Lemma~\ref{lm:equiv-i}.

\begin{lemma} \label{lm:twoequiv} Let $A$ and $B$ be subsets of  $G$. 
\begin{enumerate}
\item If $A \approx A'$ and $B \approx B'$, then $A \succsim B \implies A' \succsim B'$.\vspace{.1cm}
\item $A \approx B \implies A\sim B$.
\end{enumerate}
\end{lemma}

\noindent It follows from the lemma above that we have a natural surjection at the level of sets:
$$P_\approx(G) \to P_\sim(G), \qquad [A] \mapsto \langle A \rangle, \quad A \subseteq G.$$
This map is \emph{not order preserving} because $A\propto B$ does not imply that $A \precsim B$. It is not even true that $A \subseteq B$ implies $A \precsim B$. If $A,B \subseteq G$, then we write $[A] \sim [B]$ if $A \sim B$.

\begin{lemma} \label{lm:equiv-ii}
Let $X$ be a locally compact Hausdorff $G$-space, let $x \in X$,  let $U \subseteq X$ be a relatively compact open subset such that $\bigcup_{g \in G} g.U = X$, and let $A:= O_X(U,x)$ be as defined in \eqref{eq:O(V,x)}. Assume that $A$ is non-empty.
\begin{enumerate}
\item If $B \subseteq G$, if $\{g_\alpha\}$ is a net in $G$, and if $y \in X$ satisfies
$g_\alpha.x \to y$ and $Ag_\alpha^{-1} \to B$, then $B \approx O_X(U,y)$. \vspace{.1cm}
\item If $y \in \overline{G.x}$, then $O_X(U,y) \precsim O_X(U,x)$. \vspace{.1cm}
\item If $B\subseteq G$ is non-empty and $B \precsim O_X(U,x)$, then
there exists $y \in \overline{G.x}$ 
such that $O_X(U,y)\approx B$. 
\end{enumerate}
\end{lemma}

\begin{proof} (i). First we show that $O_X(U,y) \subseteq B$.
Let $h\in O_X(U,y)$.  Then $h.y\in U$, so  $hg_\alpha.x\in U$ eventually.
Hence $hg_\alpha\in A$ eventually, so  $h\in Ag_\alpha^{-1}$ eventually. This implies that $h \in B$.

Conversely, let $h\in B$. Then $hg_\alpha\in A$ eventually. 
It follows that $hg_\alpha.x\in U$ eventually, so $h.y$ is in $\overline{U}$. This shows that $B \subseteq O_X(\overline{U},y)$. Since $\overline{U}$ is compact, and by the assumptions on $U$, there is a finite set $F \subseteq G$ such that $\overline{U}$ is contained in $\bigcup_{g \in F} g.U = F.U$. Hence, by Lemma \ref{lm:O(V,x)-1}, 
$$B \subseteq O_X(\overline{U},y) \subseteq O_X(F.U,y) = F.O_X(U,y) \propto O_X(U,y).$$ 

(ii). There exists a net $\{g_\alpha\}$ in $G$ such that $g_\alpha.x\to y$. 
By taking a subsequence if necessary, 
we may assume that $Ag_\alpha^{-1} \to B$ for some subset $B$ of  $G$. 
It follows from (i) that $B\approx O_X(U,y)$.  Hence (ii) holds. 

(iii).
There exist a net $\{g_\alpha\}$ in $G$ and $B'\subseteq G$ 
such that $Ag_\alpha\to B' \approx B$. 
Since $B$ is non-empty, $B'$ is also non-empty. 
Take $h\in B'$. 
Then $hg_\alpha^{-1}$ is in $A$ eventually, 
and so $g_\alpha^{-1}.x \in h.U$ eventually. By passing to a subnet of $\{g_\alpha\}$ we may assume that $g_\alpha^{-1}.x \to y$ for some $y \in h.\overline{U} \subseteq X$. It now follows from (i) that $B'\approx O_X(U,y)$, 
which completes the proof. 
\end{proof}

\begin{theorem} \label{thm:basepointinvariance} Let $G$ be a countable infinite group and let $X$ be a locally compact Hausdorff space on which $G$ acts co-compactly. Suppose that $x,y \in X$ are such that $G.x$ and $G.y$ are dense in $X$. It follows that
$$\Type(G,X,x) \sim \Type(G,X,y).$$
\end{theorem}

\begin{proof} Choose a $G$-regular compact subset $K \subseteq X$. Then
$$
\Type(G,X,x) = [O_X(K^\mathrm{o},x)], \qquad \Type(G,X,y) = [O_X(K^\mathrm{o},y)].
$$
by Proposition \ref{prop:type-condition}. By the assumption that $G.x$ and $G.y$ are dense in $X$, it follows from Lemma \ref{lm:equiv-ii} (ii) that $O_X(K^\mathrm{o},x) \sim O_X(K^\mathrm{o},y)$. 
\end{proof}

\noindent Because of Theorem \ref{thm:basepointinvariance} we can make the following definition:

\begin{definition} \label{def:typeII}
Let $X$ be a locally compact Hausdorff space on which a countable infinite group $G$ acts co-compactly, and such that $G.x$ is dense in $X$ for some $x \in X$. Then set
$$\Typ(G,X) = \langle A \rangle \in P_\sim(G),$$
where $A \subseteq G$ is such that $\Type(G,X,x) = [A]$. 
\end{definition}

\begin{proposition} \label{prop:sametypei}
Let $G$ be a countable infinite group and let $X$ and $Y$ be locally compact Hausdorff spaces on which $G$ acts co-compactly with a dense orbit. If there is a surjective proper continuous $G$-map $\varphi \colon X \to Y$, then $\Typ(G,X) = \Typ(G,Y)$. 
\end{proposition}

\begin{proof} Let $x_0 \in X$ be such that $G.x_0$ is dense in $X$ and set $y_0 = \varphi(x_0)$. Being continuous and surjective, $\varphi$ maps dense sets onto dense sets, it follows that $G.y_0 = \varphi(G.x_0)$ is dense in $Y$. By Proposition \ref{prop:sametype} we have $\Type(G,X,x_0) = \Type(G,Y,y_0)$. This implies that $\Typ(G,X) = \Typ(G,Y)$. 
\end{proof}

\noindent
Recall that $A \subseteq G$ is of minimal type if $[A]$ is the type of some pointed minimal locally compact $G$-space. We give an intrinsic description of subsets of minimal type in terms of the ordering $\precsim$ on $P(G)$:

\begin{proposition} \label{prop:char-min}
Let $G$ be a countable infinite group and let $A \subseteq G$. Then $A$ is of minimal type if and only if $\langle A \rangle$ is a minimal point in the partially ordered set $(P_\sim(G), \precsim)$. 

In other words, $A$ is of minimal type if and only if whenever $B \subseteq G$ satisfies $B\precsim A$, then $B \sim A$.
\end{proposition}

\begin{proof} Suppose first that $A$ is of minimal type and let $(X,x_0)$ be a minimal pointed locally compact $G$-space with $\Type(G,X,x_0) = [A]$. Let $K \subseteq X$ be a $G$-regular compact subset of $X$. 

Suppose that $B\subseteq G$ and $B \precsim A$. It then follows from Lemma \ref{lm:equiv-ii} (iii) that $O_X(K^\mathrm{o},y) \approx B$ for some $y \in X$. Since $X$ is minimal we know that $G.y$ is dense in $X$. It therefore follows from Theorem \ref{thm:basepointinvariance} and Proposition \ref{prop:type-condition} that 
$$[A] = \Type(G,X,x_0) \sim \Type(G,X,y) = [O_X(K^\mathrm{o},y)] = [B].$$

Suppose, conversely, that $\langle A \rangle$ is a minimal point in $(P_\sim(G), \precsim)$. 
Consider the $G$-space $X_A\subseteq \beta G$, and take a minimal closed invariant subset $Z$ of $X_A$.  Take $z\in Z$. 
Set $U=K_A\cap Z$, which is a compact open subset of $Z$. 
Since the $G$-orbit of $e$ is dense in $X_A$, 
it follows from Lemma \ref{lm:equiv-ii} (ii) that 
$$A\; =\; O_{X_A}(K_A,e) \; \succsim \; O_{X_A}(K_A,z)\; =\; O_Z(U,z).$$ 
Hence $A\sim O_Z(U,z)$, because $A$ is minimal. 
By Lemma \ref{lm:equiv-ii} (iii), we can find $w\in Z$ such that $O_Z(U,w)\approx A$. It follows from Proposition~\ref{prop:type-condition} that $\Type(G,Z,w) = [A]$, so $A$ is of minimal type. 
\end{proof}

\noindent Each element of $P_\sim(G)$ dominates a minimal element, and the order relation "$\precsim$" on $P_\sim(G)$ corresponds to the order by inclusion on invariant subsets of any locally compact $G$-space:

\begin{corollary} \label{cor:orderiso} Let $G$ be a countable group, let $A$ be a non-empty subset of $G$, and let $X$ be a locally compact space on which $G$ acts co-compactly with a dense orbit such that $\Typ(G,X) = \langle A \rangle$. 

It follows that the map $Y \mapsto \Typ(G, Y)$
is an order preserving surjection from the set of closed $G$-invariant singly generated~\footnote{We say that a closed invariant subset $Y$ of a $G$-space $X$ is singly generated if $Y = \overline{G.y}$ for some $y \in X$, i.e., if there is a dense orbit.} subsets of $X$, ordered by inclusion, onto the ordered set
$$\{\alpha \in \Typ(G,X) \mid \alpha \precsim \langle A \rangle \}.$$

In particular,  for each non-empty subset $A$ of $G$ there exists a non-empty subset $B$ of $G$ of minimal type such that $B \precsim A$. 
\end{corollary}

\begin{proof} Let $G.x$ be a dense orbit in $X$. Let $K$ be a $G$-regular compact subset of $X$. Then $\Type(G,X,x) = [O_X(K^{\mathrm{o}},x)]$ by Proposition \ref{prop:type-condition}, so $O_X(K^{\mathrm{o}},x) \sim A$.

Suppose that $Y$ is a singly generated closed invariant subsets of $X$, with dense orbit $G.y$, where $y \in X=\overline{G.x}$. Then $O_X(K^{\mathrm{o}},y) \precsim O_X(K^{\mathrm{o}},x)$ by Lemma~\ref{lm:equiv-ii}~(ii).  Moreover, $K \cap Y$ is a $G$-regular compact subset of $Y$ whose relative interior is $K^{\mathrm{o}} \cap Y$; and $$O_X(K^{\mathrm{o}},y) = O_X(K^{\mathrm{o}} \cap Y,y)$$
by $G$-invariance of $Y$. Hence $\Type(G,Y,y) = [O_X(K^{\mathrm{o}},y)]$ by Proposition~\ref{prop:type-condition}, so $\Typ(G,Y) = \langle O_X(K^{\mathrm{o}},y) \rangle \precsim \langle A \rangle$. 

If $Y_1 \subseteq Y_2$ are singly generated closed invariant subsets of $X$, then $\Typ(G,Y_1) \precsim \Typ(G,Y_2)$ by the first part of this proof (with $Y_2$ in the place of $X$).

If $B \precsim A$, then there exists $y \in X$ such that $O_X(K^{\mathrm{o}},y) \approx B$ by Lemma~\ref{lm:equiv-ii}~(iii). Arguing as above, using that $K \cap \overline{G.y}$ is $G$-regular in $\overline{G.y}$, we see that 
$$\Type(G,\overline{G.y},y) =[O_X(K^{\mathrm{o}} \cap \overline{G.y},y)] = [O_X(K^{\mathrm{o}},y)] \approx B.$$
Hence $\Typ(G,\overline{G.y}) = \langle B \rangle$. 

To prove the last claim, let $A$ be any non-empty subset of $G$. Then there exists a co-compact locally compact $G$-space $X$ with $\Typ(G,X) = \langle A \rangle$, for example $X = X_A$. The  co-compact locally compact $G$-space $X$ has a minimal closed invariant subset $Y$. Let $B \subseteq G$ be such that $\Typ(G,Y) = \langle B \rangle$. Then $B$ is of minimal type and $B \precsim A$. 
\end{proof}

\noindent The example and the proposition below illustrate some properties of the the relation "$\precsim$".

\begin{example} Let $G$ be a countable group.

(i). If $A \subseteq G$, then $G \precsim A$ if and only if $A$ is equivalent to an absorbing set. 

It is not difficult to prove this directly. It also follows from Corollary~\ref{cor:orderiso}. Indeed, consider the $G$-space $X_A$ which is of type $\langle A \rangle$. Then $G \precsim A$ if and only if there is a closed $G$-invariant singly generated subset $Y$ of $X_A$ of type $\langle G \rangle$. Now, $Y$ is of type $\langle G \rangle$ if and only if $Y$ is compact, cf.\ Proposition \ref{prop:compacttype}. It follows from \cite[Proposition 5.5]{KelMonRor:supramenable} that $X_A$ has a compact minimal closed invariant subspace if and only if $A$ is equivalent to an absorbing set.

(ii). If $A \subseteq G$, then $\{e\} \precsim A$ if and only if $A$ is not infinitely divisible.

This follows from Corollary \ref{cor:orderiso} in a similar way as above. A  minimal closed $G$-invariant subset $Z \subseteq X_A$ is of type $\langle \{e\} \rangle$ if and only if $Z$ is discrete, cf.\ Proposition \ref{prop:discrete}; and  $X_A$ has a discrete minimal closed invariant subset  if and only if $A$ is not infinitely divisible by Theorem \ref{thm:discrete-vs-infdiv}.
\end{example}

\noindent It follows from the proposition below that if $A$ and $B$ are non-empty subsets of $G$ such that $A \sim B$, then $A$ is paradoxical if and only if $B$ is paradoxical.

\begin{proposition} \label{prop:paradoxical}
Let $G$ be a countable infinite group, and let $A$ be a non-empty subset of $G$. Then the following conditions are equivalent:
\begin{enumerate}
\item $A$ is non-paradoxical. \vspace{.1cm}
\item $B \precsim A$ for some non-paradoxical non-empty subset $B$ of $G$. \vspace{.1cm}
\item  Any  locally compact Hausdorff space $X$, on which $G$ acts co-compactly and with $\Typ(G,X)=\langle A \rangle$, admits a non-zero $G$-invariant Radon measure.
\end{enumerate}
\end{proposition}

\begin{proof} (i) $\Rightarrow$ (ii) is trivial (take $B = A$).

(i) $\Leftrightarrow$ (iii) follows from Corollary \ref{cor:Radon} and Definition \ref{def:typeII}.

 (ii) $\Rightarrow$ (i). Note first that if $C$ and $D$ are subsets of $G$ such that $C \approx D$, then $C$ is non-paradoxical if and only if $D$ is non-paradoxical. Indeed, if $\mu$ is any finitely additive left-invariant measure on $P(G)$, then $$0 < \mu(C) < \infty \iff 0 < \mu(D) < \infty,$$ whenever $C \approx D$.

It therefore suffices to show that if $\{g_n\}$ is a sequence in $G$ such that $Ag_n \to B$, and if $B$ is non-paradoxical, then $A$ is non-paradoxical. Let $\mu$ be a left-invariant measure on $P(G)$ such that $\mu(B)=1$.

Choose a free ultrafilter $\omega$ on $\N$. We can then take the limit of any sequence $\{C_n\}_{n=1}^\infty$ of subsets of $G$ along $\omega$ as follows:
$$\lim_\omega C_n = C \iff \forall x \in G: \lim_\omega 1_{C_n}(x) = 1_C(x).$$
If $C_n \to C$, then $\lim_\omega C_n = C$. 

Define $\phi \colon P(G) \to P(G)$ by $\phi(E) = \lim_\omega Eg_n$. In other words,
\begin{equation} \label{eq:omega}
 \forall E \subseteq G \; \, \forall x \in G : 1_{\phi(E)}(x) = \lim_\omega 1_{Eg_n}(x).
\end{equation}
Then $\phi(A) = B$. One can use \eqref{eq:omega} to verify that the following holds:
\begin{itemize}
\item $\phi(\emptyset) = \emptyset$. \vspace{.1cm}
\item $\phi(C \cup D) = \phi(C) \cup \phi(D)$ for all $C,D \subseteq G$. \vspace{.1cm}
\item $\phi(C \cap D) = \phi(C) \cap \phi(D)$ for all $C,D \subseteq G$. \vspace{.1cm}
\item $\phi(gC)  = g \, \phi(C)$ for all $g \in G$ and $C \subseteq G$.
\end{itemize}
It follows that $\nu:= \mu \circ \phi$ is a finitely additive left-invariant measure on $P(G)$ with $\nu(A) = \mu(\phi(A)) =\mu(B)=1$. Hence (i) holds.
\end{proof}

\noindent With the base point free notion of type we can reinterpret Definition~\ref{def:min-universal} as follows:

\begin{proposition} \label{lm:universal-2}
Let $G$ be a countable infinite group and let $A$ be a non-empty subset of $G$ of minimal type. A locally compact minimal $G$-space $Z$ is then a universal minimal $G$-space of type $A$ (in the sense of Definition~\ref{def:min-universal}) if and only if for each minimal locally compact $G$-space $X$ with $\Typ(G,X) = \Typ(G,Z)$ there is a surjective proper continuous $G$-map $\varphi \colon Z \to X$. 
\end{proposition}

\noindent Because of this result it makes sense to refer to $Z$ as being a \emph{universal minimal $G$-space} (without reference to the subset $A$ of $G$), as the class of this space is contained in $Z$. 

\begin{proof} 
Choose  $x_0 \in X$, and let $B \subseteq G$ be such that $\Type(G,X,x_0) = [B]$. Then $A \sim B$ because $\Typ(G,X) = \langle A \rangle$. Let $K$ be a $G$-regular compact subset of $X$. Then  $[B] = [O_X(K^{\mathrm{o}},x_0)]$ by Proposition \ref{prop:type-condition}. It follows in particular that $A \precsim O_X(K^{\mathrm{o}},x_0)$.  We can therefore use Lemma \ref{lm:equiv-ii} (iii) to find $x_1$ in $X$ such that $O_X(K^{\mathrm{o}},x_1) \approx A$. Minimality of $X$ ensures that $G.x_1$ is dense in $X$. Hence $\Type(G,X,x_1) = [A]$ by Proposition \ref{prop:type-condition}. The existence of a surjective proper continuous $G$-map $\varphi \colon Z \to X$ now follows from the universal property of $Z$.

Conversely,  if $\varphi \colon Z \to X$ is a proper continuous $G$-map, then $\Typ(G,X) = \Typ(G,Z) = \langle A \rangle$ by Proposition~\ref{prop:sametypei}.
\end{proof}

\noindent Recall from Proposition~\ref{prop:universal_minimal} that every closed minimal $G$-invariant subset of $X_A$ is a universal minimal $G$-space of type $A$, whenever $A$ is a subset of $G$ of minimal type. Hence universal minimal $G$-spaces always exist, and they are also unique (by Theorem~\ref{thm:unique-minimal}). 

\begin{corollary} \label{cor:char-same-type}
Let $(X_i)_{i \in I}$ be a family of minimal locally compact $G$-spaces. Then the following conditions are equivalent:
\begin{enumerate}
\item $\Typ(G,X_i) = \Typ(G,X_j)$ for all $i,j \in I$. \vspace{.1cm}
\item There exist a minimal locally compact $G$-space $Z$ and a (surjective) proper continuous $G$-map $\varphi_i \colon Z \to X_i$ for each $i \in I$. \vspace{0.1cm}
\item There exists $x_i \in X_i$ for each $i \in I$ such that $$\Type(G,X_i,x_i) = \Type(G,X_j,x_j)$$
for all $i,j \in I$. 
\end{enumerate}
\end{corollary}

\begin{proof}  (i) $\Rightarrow$ (ii). Fix $i_0 \in I$ and let $A$ be a (necessarily minimal) non-empty subset of $G$ such that $\Typ(G,X_{i_0}) = \langle A \rangle$. Take a universal minimal $G$-space  $Z$ of type $A$, cf.\ Definition~\ref{def:min-universal} and Proposition~\ref{prop:universal_minimal}.  Then $\Typ(G,Z) = \langle A \rangle = \Typ(G,X_i)$ for all $i \in I$ by Lemma~\ref{lm:type-of-universal} and Definition~\ref{def:typeII}, and so it follows from Proposition~\ref{lm:universal-2} that $Z$ satisfies the conditions in (ii).

 (ii) $\Rightarrow$ (iii). Suppose that $Z$ and $\varphi_i \colon Z \to X_i$ are as in (ii). Pick any point $z_0 \in Z$, and set $x_i = \varphi(z_0)$. Then 
$$\Type(G,X_i,x_i) = \Type(G,Z,z_0)$$
by Proposition~\ref{prop:sametype}. 

 (iii) $\Rightarrow$ (i) follows from Definition~\ref{def:typeII} (together with Theorem~\ref{thm:basepointinvariance}). 
\end{proof}

\noindent The corollary above is stated in a rather general form, but is perhaps of particular interest when the family $(X_i)_{i \in I}$ has just two elements. Consider the relation on minimal locally compact $G$-spaces $X_1$ and $X_2$, that there exists a third $G$-space $Z$ that maps surjectively onto $X_1$ and $X_2$ by continuous $G$-maps. By the corollary, this happens if and only if $\Typ(G,X_1) = \Typ(G,X_2)$. In particular, this defines an equivalence relation on the class of minimal locally compact $G$-spaces. A priori, it was not clear to the authors that this relation is transitive. 

We end by stating two questions:

\begin{question} \label{q1}
Given a non-empty subset $A$  (not necessarily of minimal type) of a (countable infinite) group $G$. Is it true that each minimal closed $G$-invariant subset of $X_A$ is a universal minimal $G$-space (in the sense defined below Proposition~\ref{lm:universal-2})? 
\end{question}

\noindent In other words, are all minimal closed $G$-invariant subsets of any co-compact open $G$-invariant subset of $\beta G$ a universal space? If yes, then these spaces are classified by their type by Theorem~\ref{thm:unique-minimal}.

\begin{question} \label{q2}
Given a non-empty subset $A$ (not necessarily of minimal type) of a (countable infinite) group $G$. Classify the (base point free) types of the minimal closed $G$-invariant subsets of $X_A$ in terms of the set $A$.
\end{question}

\noindent If $A$ is of minimal type, then each minimal closed $G$-invariant subsets of $X_A$ also has type $A$ (by Proposition~\ref{prop:universal_minimal}). We also know when $X_A$ has minimal closed $G$-invariant subspaces that are discrete (of type $\langle \{e\} \rangle$), respectively, compact (of type $\langle G \rangle$). 

{\small{
\bibliographystyle{amsplain}
\providecommand{\bysame}{\leavevmode\hbox to3em{\hrulefill}\thinspace}
\providecommand{\MR}{\relax\ifhmode\unskip\space\fi MR }
\providecommand{\MRhref}[2]{%
  \href{http://www.ams.org/mathscinet-getitem?mr=#1}{#2}
}
\providecommand{\href}[2]{#2}

}

\end{document}